\documentclass[11pt]{article}

\usepackage{amsmath, amsthm, amssymb, amscd}
\usepackage[title]{appendix}

\usepackage{hyperref}

\newcommand{\R}{\mathbb{R}}

\numberwithin{equation}{section}
\newtheorem{thm}{Theorem}
\newtheorem{lem}{Lemma}
\newtheorem{rem}{Remark}

\newcommand{\MLE}{\operatorname{MLE}}
\renewcommand{\P}{\mathbb{P}}

\newcommand{\supp}{\operatorname{supp}}

\newcommand{\KL}{\operatorname{KL}}

\newcommand{\mH}{\mathcal{H}}

\begin{document}

\title{Minimax theory for a class of non-linear statistical inverse problems}

\author{Kolyan Ray\footnote{The research leading to these results has received funding from the European Research Council under ERC Grant Agreement 320637.\newline Email: \href{mailto:k.m.ray@math.leidenuniv.nl}{k.m.ray@math.leidenuniv.nl}, \href{mailto:schmidthieberaj@math.leidenuniv.nl}{schmidthieberaj@math.leidenuniv.nl}}\; and Johannes Schmidt-Hieber
 \vspace{0.1cm} \\
 {\em Leiden University} }

\date{}
\maketitle

\begin{abstract}
We study a class of statistical inverse problems with non-linear pointwise operators motivated by concrete statistical applications. A two-step procedure is proposed, where the first step smoothes the data and inverts the non-linearity. This reduces the initial non-linear problem to a linear inverse problem with deterministic noise, which is then solved in a second step. The noise reduction step is based on wavelet thresholding and is shown to be minimax optimal (up to logarithmic factors) in a pointwise function-dependent sense. Our analysis is based on a modified notion of H\"older smoothness scales that are natural in this setting.
\end{abstract}

%

\section{Introduction}

We study the minimax estimation theory for a class of non-linear statistical inverse problems where we observe the path $(Y_t)_{t\in [0,1]}$ arising from
\begin{align}
	dY_t = (h\circ K f ) (t) dt + n^{-1/2} dW_t, \quad t\in [0,1].
	\label{eq.mod_link=h}
\end{align}
Here, $h$ is a known strictly monotone link function, $K$ is a known linear operator mapping $f$ into a univariate function and $W=(W_t)_{t\in [0,1]}$ denotes a standard Brownian motion. The number $n$ plays the role of sample size and asymptotic statements refer to $n\rightarrow \infty.$ The statistical task is to estimate the regression function $f$ from data $(Y_t)_{t\in [0,1]}$. Model \eqref{eq.mod_link=h} includes the classical linear statistical inverse problem as a special case for the link function $h(x)=x.$ 

Various non-Gaussian statistical inverse problems can be rewritten in the form \eqref{eq.mod_link=h}. A strong notion to make this rigorous is to prove that the original model converges to a model of type \eqref{eq.mod_link=h} with respect to the Le Cam distance. This roughly means that for large sample sizes, such a model transformation 	does not lead to a loss of information about $f.$ Most such results are established in the direct case where $K$ is the identity operator, but as mentioned in \cite{RaySchmidt-Hieber2015b}, such statements can often be extended to the inverse problem setting (see \cite{meister2011} for an example). A particularly interesting example covered within this framework is Poisson intensity estimation, where we observe a Poisson process on $[0,1]$ with intensity $n Kf.$ This can be rewritten in the form \eqref{eq.mod_link=h} with $h(x) =2 \sqrt{x},$ see Section \ref{sec.specific_cases}. Reconstructing information from Poisson data has received a lot of attention in the literature due to its applications in photonic imaging, see Vardi et {\it al.} \cite{vardi1985}, Cavalier and Koo \cite{cavalier2002}, Hohage and Werner \cite{hohage2013}, and Bertero et {\it al.} \cite{bertero2009}. Other examples of models of type \eqref{eq.mod_link=h} include density estimation, density deconvolution, spectral density estimation, variance estimation, binary regression and functional linear regression. For further details on the approximation of models and additional examples, see Section \ref{sec.specific_cases}.

For linear link functions, we recover the well-known linear statistical inverse problem. In this case, the classical approach is to project the data onto the eigenbasis of $K^*K$ and to use the first (weighted) empirical basis coefficients to reconstruct the signal (cf. the survey paper Cavalier \cite{cavalier2008}). For non-linear link functions, spectral methods cannot be applied to obtain rate optimal estimators and we propose to pre-smooth the data instead. Noise reduction as a pre-processing step for inverse problems has been proposed in different contexts by Bissantz et {\it al.} \cite{bissantz2004}, Klann et {\it al.} \cite{klann2006}, Klann and Ramlau \cite{klann2008} and Math\' e and Tautenhahn \cite{mathe2011a, mathe2011b}. Let us briefly sketch the idea in the linear case $h(x)=x.$ Suppose that the signal $f$ has H\"older (or Sobolev) regularity $\beta$ and that $K$ has degree of ill-posedness $s$, mapping $\beta$-smooth signals into a space of $(\beta+s)$-smooth  functions. The minimax rate (possibly up to $\log n$-factors) for estimation of $Kf$ from data $(Y_t)_{t\in [0,1]}$ under $L^p$-loss is then $n^{-\frac{\beta+s}{2\beta+2s+1}}.$ Using pre-smoothing, we thus have access to
\begin{align}
Y^\delta_t= Kf(t) + \delta(t),
\label{eq.det_inv_prob}
\end{align}
where $\|\delta\|_{L^p} \lesssim n^{-\frac{\beta+s}{2\beta+2s+1}}$ with high probability. Reconstruction of $f$ from $Y^\delta$ may be now viewed as a deterministic inverse problem with noise level $\delta.$ In such a deterministic setup, the regression function $f$ can be estimated at rate $\|\delta\|_{L^p}^{\frac{\beta}{\beta+s}}\lesssim n^{-\frac{\beta}{2\beta+2s+1}}$ (cf. Natterer \cite{natterer1984}, Engl et {\it al.} \cite{engl1996}, Section 8.5), which coincides with the minimax rate of estimation (cf. \cite{cavalier2008}). Reducing the original statistical inverse problem to a deterministic inverse problem via pre-smoothing of the data therefore leads to rate-optimal procedures. The advantage of pre-smoothing the data is that it is quite robust to all sorts of model misspecifications, in particular outliers in the data. A drawback of the method is of course that for the second step, we need to know the order of the noise level $\delta$ and thus the smoothness of $Kf.$ We discuss various approaches to this in Section \ref{sec.extension}.

Following this two-step strategy, the statistical problem reduces to finding error bounds in the direct problem (i.e. $K$ is the identity). To pre-smooth the data, we apply the following general plug-in principle: construct an estimator $\widehat{h\circ Kf}$ of $h\circ Kf$ by hard wavelet thresholding and take $h^{-1}(\widehat{h\circ Kf})$ (recall that $h$ is injective) to be the pre-smoothed version $Y^\delta$ of $Kf.$ Analyzing this method for non-linear link functions $h$ is a challenging problem, since both the amount of smoothing and the minimax estimation rates are very sensitive to the local regularity properties of $h.$ Due to this individual behaviour depending on the non-linear link functions, we have restricted ourselves to those that are most relevant from a statistical point of view, namely $h(x)$ equal to $2\sqrt{x},$ $2^{-1/2}\log x$ and $2\arcsin\sqrt{x}$ (see Section \ref{sec.specific_cases} for more details). While we take a similar overall approach for these three link functions, the different natures of their singular points mean that the details of each case must be treated separately. Indeed, we believe that there is no feasible proof that covers them all at once.

We remark that two link functions $h$, $\widetilde h$ in \eqref{eq.mod_link=h} can be considered equivalent if they can be transformed into one another via $\mathcal{C}^\infty$-functions, in which case the minimax rates of estimation are identical up to constants. Since we concern ourselves with the rate and not the minimax constants, our results apply to any $\widetilde h$ that are equivalent to the specific link functions we study. In particular, the constants in these three cases have no impact.

We seek to control the difference between the pre-smoothed data $Y^\delta$ and $Kf,$ that is the noise level $\delta$ in the deterministic inverse problem \eqref{eq.det_inv_prob}. Due to the non-linearity of $h,$ the rates for $\delta$ are highly local and we establish matching pointwise upper and lower bounds  for $|\delta(t)|$ that depend on the smoothness of $Kf$ and the value $|Kf(t)|.$ From this, bounds in $L^p$  for all $1 \leq p \leq \infty$ can be deduced. Local rates of this type allow one to capture spatial heterogeneity of the function $f$ and were considered only recently for density estimation \cite{patschkowski2016}, which is related to model \eqref{eq.mod_link=h} with link function $h(x)=2\sqrt{x}.$

Surprisingly, in each of the three cases considered here there are two regimes for the local convergence rate. We call $x$ an \textit{irregular point} of the link function $h$ if $h'$ is singular at $x.$ For example, the link function $h(x)=2\sqrt{x}$ has irregular point $x=0.$ Away from any irregular point of $h,$ we obtain a local rate corresponding to the ``classical" nonparametric minimax rate. However, when close to an irregular point we obtain a different type of convergence rate which is related to estimation in irregular models. For the direct problem, we show that the derived rates are minimax optimal up to $\log n$-factors.

When dealing with non-linear link functions, classical H\"older smoothness spaces can only be employed up to a certain degree. For instance, it is well-known that if $f$ has H\"older smoothness $\beta \leq 1$ then $\sqrt{f}$ has H\"older smoothness $\beta/2$. However, this breaks down for the usual definition of H\"older smoothness when $\beta>1$ (cf. \cite{Bony2006}). So far, the common approach to this issue (within for example density estimation \cite{nussbaum1996}) is to assume that $f$ is uniformly bounded away from zero, which is an artificial restriction from a modelling perspective. There has been recent progress in removing this constraint \cite{patschkowski2016}, though such results can not be extended beyond $\beta=2$. We take a different approach here and consider a suitable modification of H\"older scales \cite{RaySchmidt-Hieber2015c} for non-negative functions which maintains this relationship. This framework is natural in our setting and is expanded upon in Section \ref{sec.func_spaces}. 

The article is organized as follows. The modified H\"older spaces are introduced in Section \ref{sec.func_spaces}, the main upper and lower bound results are presented in Section \ref{sec.main_results} with further discussion on extending this approach to the full inverse problem setting in Section \ref{sec.extension}. Section \ref{sec.specific_cases} provides additional motivation for model \eqref{eq.mod_link=h} and examples of statistical problems that can be rewritten in this form. Proofs and technical results are deferred to Sections \ref{sec.proofs} and \ref{sec.technical_results} in the appendix.

{\it Notation:} For two sequences $(a_n)$, $(b_n)$ of numbers, we write $a_n \asymp b_n$ if $|a_n/b_n|$ is bounded away from zero and infinity as $n\rightarrow\infty$. For two numbers $a,b$ arbitrarily selected from some set $S$, we also write $a \asymp b$ if $|a/b|$ is bounded away from zero and infinity for all $a,b\in S$ ($S$ will always be clear in the context). The maximum and minimum of two real numbers are denoted by $a\vee b$ and $a\wedge b$, while $\lfloor a \rfloor$ is the greatest integer strictly smaller than $a$. Let $KL(f,g) = \int f\log(f/g)d\mu$ be the Kullback-Leibler divergence for non-negative densities $f$, $g$ relative to a dominating measure $\mu$. Recall that in the Gaussian white noise model \eqref{eq.mod_link=h}, the Kullback-Leibler divergence between two measures generating such processes with regression functions $f$, $g$ equals $KL(P_f,P_g) = \frac{n}{2}\|h\circ Kf-h\circ Kg\|_2^2$.

\section{H\"older spaces of flat functions}
\label{sec.func_spaces}

In this section, we consider a restricted H\"older class of functions whose derivatives behave regularly in terms of the function values. This class has better properties under certain pointwise non-linear transformations compared to classical H\"older spaces (cf. \cite{RaySchmidt-Hieber2015c}) and is therefore appropriate as a parameter space for adaptive estimation in model \eqref{eq.mod_link=h}. In particular, it allows one to express pointwise convergence rates in terms of the function values alone.

Let $|f|_{C^\beta} = \sup_{x\neq y, x,y\in [0,1]} |f^{(\lfloor \beta \rfloor)}(x) - f^{(\lfloor \beta \rfloor)}(y)| /|x-y|^{\beta - \lfloor \beta \rfloor}$ be the usual H\"older seminorm and consider the space of $\beta$-H\"older continuous functions on $[0,1]$,
\begin{equation*}
C^\beta = \big\{ f: [0,1] \rightarrow \R \  : \  f^{(\lfloor \beta \rfloor)} \text{ exists}, \  \| f \|_{C^\beta} := \| f\|_\infty + \| f^{(\lfloor\beta\rfloor)} \|_\infty + |f|_{C^\beta}< \infty \big\},
\end{equation*}
where $\|\cdot\|_\infty$ denotes the $L^\infty[0,1]$-norm. Define the following seminorm on $C^\beta$:
\begin{align*}
| f |_{\mathcal{H}^\beta} = \max_{1 \leq j <\beta } \left( \sup_{x\in[0,1]} \frac{|f^{(j)}(x)|^\beta}{|f(x)|^{\beta-j}} \right)^{1/j}
= \max_{1 \leq j <\beta } \Big\| |f^{(j)}|^\beta/|f|^{\beta-j} \Big\|_\infty^{1/j}
\end{align*}
with $0/0$ defined as $0$ and $|f|_{\mH^\beta}=0$ for $\beta \leq 1.$ The quantity $| f |_{\mathcal{H}^\beta}$ measures the flatness of a function near zero in the sense that if $f(x)$ is small, then the derivatives of $f$ must also be small in a neighbourhood of $x$. Let
\begin{align*}
	\| f\|_{\mH^\beta} = \|f\|_{C^\beta} + |f|_{\mH^\beta}
\end{align*}
and consider the space of non-negative functions 
\begin{align*}
	\mathcal{H}^\beta = \{ f\in C^\beta \ : \  f\geq 0, \ \|f\|_{\mH^\beta}<\infty\}.
\end{align*}
This space contains for instance the constant functions, all $C^\beta$ functions that are uniformly bounded away from zero and any function of the form $f(x) = (x-x_0)^\beta g(x)$ for $g \geq \varepsilon > 0$ infinitely differentiable. Due to the non-negativity constraint, $\mH^{\beta}$ is no longer a vector space. However, it can be shown that it is a convex cone, that is, for any $f,g \in \mH^{\beta}$ and positive weights $\lambda, \mu>0$ then also $\lambda f +\mu g \in  \mH^{\beta}.$ Moreover, $\|\cdot\|_{\mH^\beta}$ defines a norm in the sense of Theorem 2.1 in \cite{RaySchmidt-Hieber2015c}. For $0 < \beta \leq 2$, the additional derivative constraint is in fact always satisfied: $\mathcal{H}^\beta$ contains all non-negative functions that can be extended to a $\beta$-H\"older function on $\R$ (Theorem 2.4 of \cite{RaySchmidt-Hieber2015c}).

The space $\mathcal{H}^\beta$ allows one to relate the smoothness of $f$ to that of $h\circ f$ in a natural way for certain non-linear link functions $h$ of interest (see Lemmas \ref{lem.arcsin_lip} and \ref{lem.log_lip}). In contrast, the classical H\"older spaces $C^\beta$ behave poorly in this respect. For example, if $h(x) = 2\sqrt{x}$ then there exist infinitely differentiable functions $f$ such that $h\circ f$ is not in $C^\beta$ for any $\beta >1$ (Theorem 2.1 of \cite{Bony2006}). Based on observing \eqref{eq.mod_link=h} (with $K$ the identity), one can not exploit higher order classical H\"older smoothness of $f$; this is a fundamental limitation of the square-root transform. An alternative approach to exploit extra regularity would be to impose stronger flatness-type assumptions on only the local minima of $f$, in particular forcing them all to be zero (e.g. \cite{Bony2006b,Bony2010}). However, the approach taken here is more flexible, permitting functions to take small non-zero values.

For the Bernoulli case \eqref{eq.mod_link=arcsin} with $h(x)= 2\arcsin \sqrt{x}$, we must make a further restriction since the function range is limited to $[0,1]$:
\begin{align*}
\mathcal{H}_B^\beta = \left\{ f : [0,1] \rightarrow [0,1] : \|f\|_{\mH_B^\beta} := \|f\|_{C^\beta([0,1])} + |f|_{\mH^\beta} + |1-f|_{\mH^\beta} < \infty \right\} .
\end{align*}
In particular, this implies that there exists $\kappa>0$ such that $|f^{(j)}(x)|^\beta \leq \kappa^j \min(f(x),1-f(x))^{\beta-j}$ for $j=1,...,\lfloor \beta \rfloor$. The class $\mathcal{H}_B^\beta$ simply ensures the same behaviour near the boundary value 1 as $\mathcal{H}^\beta$ does near 0, since both $0$ and $1$ are irregular points of $h(x) = 2 \arcsin \sqrt{x}$.

\section{Main results for the direct case}
\label{sec.main_results}

In this section we derive matching upper and lower bounds for pointwise estimation in the pre-smoothing step of the full statistical model \eqref{eq.mod_link=h}. We are thus concerned with the pointwise recovery of $Kf$ from observation \eqref{eq.mod_link=h} or equivalently obtaining pointwise bounds for the noise level $\delta$ in the deterministic inverse problem \eqref{eq.det_inv_prob}. For this it suffices to consider the direct model where we observe the path $(Y_t)_{t\in [0,1]}$ with
\begin{align}
	dY_t = (h \circ f)(t) dt + n^{-\frac 12} dW_t, \quad t\in [0,1],
	\label{eq.mod_link=h_direct}
\end{align}
where $h:\R \rightarrow \R$ is the known strictly monotone link function.

Our estimation procedure is based on a wavelet decomposition of $(Y_t)_{t\in [0,1]}$. Let $\phi, \psi: \R \rightarrow \R$ be a bounded, compactly supported $S$-regular scaling and wavelet function respectively, where $S$ is a positive integer (for a general overview and more details see e.g. \cite{Hardle1998}). The scaled and dilated functions $\phi_k, \psi_{jk}$ with support intersecting the interval $[0,1]$ can be modified to form an $S$-regular orthonormal wavelet basis of $L^2([0,1])$ by correcting for boundary effects as explained in Theorem 4.4 of \cite{Cohen1993}. This basis can be written as $\{\phi_k: k\in I_{-1}\} \cup \{\psi_{j,k} : j = 0,1,...; k \in I_j\} = \{\psi_{j,k} : j = -1,0,...; k \in I_j\}$ with $I_j$ the set of wavelet basis functions at resolution level $j$ and $\psi_{-1,k}:=\phi_k.$ Due to the compact support of $\psi,$ the cardinality of $I_j$ is bounded by a constant multiple of $2^j.$

By projecting the wavelet basis onto \eqref{eq.mod_link=h_direct}, it is statistically equivalent to observe
\begin{align}\label{eq.seq_space}
Y_{j,k} := \int_0^1 \psi_{j,k}(t) dY_t = d_{j,k} + n^{-1/2} Z_{j,k}, \quad \quad \quad j = -1,0,1,..., \; k\in I_j,
\end{align}
where $d_{j,k} = \int_0^1 h(f(x)) \psi_{j,k}(x) dx$ and $(Z_{j,k})$ are i.i.d. standard normal random variables.

We deal specifically with three cases \eqref{eq.mod_link=sqrt}, \eqref{eq.mod_link=arcsin} and \eqref{eq.mod_link=log} below. In each case, we construct an estimator using a two-stage procedure involving thresholding the empirical wavelet coefficients $Y_{j,k}$ in \eqref{eq.seq_space}. In the Gaussian variance case $h(x) = 2^{-1/2} \log x$ an additional step is needed, see Section \ref{sec.log_case}.
\begin{enumerate}
\item Let $2^{J_n}$ satisfy $n/2 \leq 2^{J_n} \leq n$ and estimate $h\circ f$ by the hard wavelet thresholding estimator
\begin{equation}\label{eq.hard_wav_est}
\widehat{h} = \sum_{j = -1}^{J_n} \sum_{k \in I_j} \widehat{d}_{j,k} \psi_{j,k} = \sum_{j = -1}^{J_n} \sum_{k \in I_j}   Y_{j,k} \mathbf{1}(|Y_{j,k}| > \tau \sqrt{\log n/n}) \psi_{j,k}
\end{equation}
for some $\tau > 2\sqrt{2}$.
\item Consider the plug-in estimator $\widehat{f} = h^{-1}\circ \widehat{h}$.
\end{enumerate}

\subsection{The Poisson case $h(x) = 2\sqrt{x}$}
\label{sec.Poisson_case}

For the link function $h(x) = 2\sqrt{x}$, model \eqref{eq.mod_link=h_direct} can be written
\begin{align}
	dY_t = 2\sqrt{f(t)} dt + n^{-1/2} dW_t, \quad t\in [0,1].
	\label{eq.mod_link=sqrt}
\end{align}
This model corresponds to nonparametric density estimation \cite{nussbaum1996,brown2004,RaySchmidt-Hieber2015b}, Poisson intensity estimation \cite{brown2004} and certain nonparametric generalized linear models with Poisson noise \cite{Grama1998} (see Section \ref{sec.specific_cases} for more details). In model \eqref{eq.mod_link=sqrt} we show that the rate of estimation for $f \in \mathcal{H}^\beta$ at a point $x \in [0,1]$ is
\begin{align}
	\tilde{r}_{n,\beta}^P \big(f(x)\big) = \left( \frac{\log n}{n} \right)^{\frac{\beta}{\beta+1}} \vee \left( f(x)\frac{\log n}{n} \right)^{\frac{\beta}{2\beta+1}} ,
	\label{eq.local_rate_sqrt}
\end{align}
giving two regimes depending on the size of $f(x)$. Note that the transition between these two regimes occurs at the threshold $f(x) \asymp (\log n/n)^{\beta/(\beta+1)}$. The presence of two regimes is caused by the non-linearity of $h$, which is smooth away from 0 but has irregular point 0. Indeed, the smoothness of $h\circ f$ bears much less resemblance to that of $f$ near 0 (see Lemma 5.3 of \cite{RaySchmidt-Hieber2015c}).

The first regime occurs for small values of $f$ and gives convergence rates which are faster than the parametric rate for $\beta>1.$ For small densities, the variance strictly dominates the bias of the estimator so that the rate is purely driven by stochastic error, hence the lack of dependence on $f(x)$. The second regime reflects the ``standard" nonparametric regime and yields the usual nonparametric rate, albeit with pointwise dependence. Recall that $\tau$ is the threshold parameter of the wavelet thresholding estimator.

\begin{thm}\label{thm.est_link=sqrt}
For any $0 < \beta_1 \leq \beta_2 < \infty$ and $0< R_1 \leq R_2 < \infty$, the estimator $\widehat{f}$ satisfies
\begin{equation*}
\inf_{\substack{\beta_1 \leq \beta \leq \beta_2\\ R_1 \leq R \leq R_2}} \inf_{f : \|f\|_{\mH^\beta}\leq R} \P_f \left( \sup_{x \in [0,1]} \frac{|\widehat{f}(x) - f(x)|}{\tilde{r}_{n,\beta}^P(f(x))} \leq C(\beta_1,\beta_2,R_1,R_2) \right) \geq 1 - \frac{3n^{-(\tau^2/8-1)} }{\sqrt{2\log n}} ,
\end{equation*}
where $\tilde{r}_{n,\beta}^P(f(x))$ is given in \eqref{eq.local_rate_sqrt}.
\end{thm}
Thus, with high probability, we have that $|\widehat{f}(x)-f(x)|\leq C\tilde{r}_{n,\beta}^P(f(x))$ for all $x\in [0,1]$ and some constant $C.$ The estimator is fully adaptive, both spatially and over H\"older smoothness classes. Since the constant $C$ in Theorem \ref{thm.est_link=sqrt} does not depend on $x$, we also obtain convergence rates in $L^p$ for all $1 \leq p < \infty$,
\begin{align*}
\| \widehat{f} - f\|_{L^p} \lesssim \left( \int_0^1 \tilde{r}_{n,\beta}^P(f(x))^p dx \right)^{1/p} \lesssim (\log n/n)^{\frac{\beta}{\beta+1}} + \| f\|_{L^p}(\log n/n)^{\frac{\beta}{2\beta+1}},
\end{align*}
which are optimal up to $\log n$ factors.

The rate \eqref{eq.local_rate_sqrt} is the same (up to logarithmic factors) as that attained in density estimation \cite{patschkowski2016}, which at first sight might be expected due to the well-known asymptotic equivalence of density estimation and the Gaussian white noise model \eqref{eq.mod_link=sqrt}, cf. \cite{nussbaum1996,brown2004}. However the situation is in fact more subtle, with asymptotic equivalence only holding in the second regime of \eqref{eq.local_rate_sqrt} and under further conditions (see \cite{RaySchmidt-Hieber2015b} for more details).

For density estimation, local rates of the type \eqref{eq.local_rate_sqrt} have been obtained only up to $\beta = 2$ \cite{patschkowski2016}. This is consistent with our results since $\mH^\beta$  and the space of all functions on $[0,1]$ that can be extended to a non-negative $\beta$-H\"older function on $\R$ in fact coincide for $0<\beta\leq 2$ (Theorem 2.4 of \cite{RaySchmidt-Hieber2015c}). The authors in \cite{patschkowski2016} deal with higher derivatives in one specific situation, namely points near the support boundary, where the function necessarily satisfies a flatness type condition allowing one to quantify the behaviour of the higher order derivatives.

For practical applications, it may be useful to consider a bias corrected version of the estimator $\widehat f$,
\begin{align*}
	\widetilde f = \widehat f - \frac{1}{4n} \sum_{j=-1}^{J_n} \sum_{k\in I_j} \psi_{j,k}^2 \mathbf{1}(|Y_{j,k}| > \tau \sqrt{\log n/n}).
\end{align*}
To see this, recall that the inverse link function is $h^{-1}(y) =y^2/4$ and $E[Y_{j,k}^2] = d_{j,k}^2 +n^{-1}.$ For each selected empirical wavelet coefficient the squared noise therefore induces an upwards bias of size $n^{-1}$, which we correct by considering $\widetilde f.$

The rate \eqref{eq.local_rate_sqrt} is an upper bound for the pointwise rate of estimation of $f$ and we have a corresponding lower bound
\begin{align*}
	r_{n,\beta}^P(f(x)) = n^{-\frac{\beta}{\beta+1}} \vee ( f(x)/n )^{\frac{\beta}{2\beta+1}} ,
\end{align*}
which agrees with $\tilde{r}_{n,\beta}^P(f(x))$ in \eqref{eq.local_rate_sqrt} up to logarithmic factors.
\begin{thm}\label{thm.lower_bound_sqrt}
For any $\beta>0$, $R>0$, $x_0 \in [0,1]$ and any sequence $(f_n^*)_n$ with $\limsup_{n \rightarrow \infty}\|f_n^*\|_{\mH^\beta}<R$,
\begin{align*}
\liminf_{n \rightarrow \infty}  \inf_{\hat{f}_n(x_0)} \sup_{\substack{f:\|f\|_{\mH^\beta}\leq R \\ \KL(f,f_n^*)\leq1 }}    \P_f \left(\frac{|\hat{f}_n(x_0) - f(x_0)|}{r_{n,\beta}^P (f(x_0))} \geq C(\beta,R) \right) > 0 ,
\end{align*}
where the infimum is taken over all measurable estimators of $f(x_0)$.
\end{thm}
In the Gaussian white noise model, it is well-known that the minimax estimation rate for pointwise loss is $n^{-\frac{\beta}{2\beta+1}}$, while the adaptive estimation rate is $(n/\log n)^{-\frac{\beta}{2\beta+1}}$ (cf. \cite{lepskii1991} and Theorem 3 in \cite{brown1996}). Adaptation to the smoothness index $\beta$ therefore imposes an additional $\log n$-factor that is unavoidable in the convergence rate. We believe that a similar phenomenon leads to the additional $\log n$-terms in the upper bound and that this rate is sharp. 

In Theorem \ref{thm.lower_bound_sqrt}, we require the sequence $(f_n^*)_n$ to be strictly in the interior of $\{f:\|f\|_{\mH^\beta}\leq R\}$. The proof of Theorem \ref{thm.lower_bound_sqrt} gives some insight into the rates. In case \eqref{eq.mod_link=sqrt}, the Kullback-Leibler divergence corresponds to the Hellinger distance. Away from 0, this behaves like the $L^2$-distance, thereby leading to the usual classical nonparametric rate. As the function approaches zero however, testing between two functions depends on the $L^1$-distance and we therefore obtain the same rates as for irregular models (cf. \cite{jirak2014}). 

The statement of Theorem \ref{thm.lower_bound_sqrt} is considerably stronger than standard minimax lower bounds as it holds for all local parameter spaces $\{f: \KL(f,f_n^*)\leq1\}$ with $(f_n^*)_n$ an arbitrary sequence in the interior of the parameter space. To see the difference, consider a model $(P_f^n: f\in \Theta)$ and suppose we want to estimate $f$ with respect to a loss $\ell.$ Recall that any rate $\epsilon_n$ is a lower bound if there are two sequences of parameters $(f_{0,n})_n, (f_{1,n})_n \subset \Theta$ such that 
\begin{align}
	\ell(f_{0,n},f_{1,n})> 2\epsilon_n \quad \text{and} \ \ \KL(P_{f_{1,n}}^n,P_{f_{0,n}}^n)\leq 1
	\label{eq.lb_std_conditions}
\end{align}
(cf. \cite{tsybakov2009}, Theorem 2.2 (iii)). This is quite a weak notion, since it says that {\it somewhere} on the parameter space the estimation rate $\epsilon_n$ cannot be improved. We would of course prefer to have a notion for which the rate holds {\it everywhere} on $\Theta.$ In Theorem \ref{thm.lower_bound_sqrt} we require that $\epsilon_n$ is a lower bound on $\{f\in \Theta: KL(P_f^n, P_{f_n^*}^n)\leq 1\}$ for arbitrary sequences $(f_n^*)$. Taking $f_n^*=f_{0,n},$ this implies that for any sequence $(f_{0,n})_n\subset  \Theta$ we can find a sequence $(f_{1,n})_n \subset \Theta$ satisfying \eqref{eq.lb_std_conditions}. 

This change of definition makes a big difference in terms of the  lower bounds obtained. It is not hard to see for instance that $n^{-\frac{\beta}{2\beta+1}}$ is a minimax lower bound over an $\mH^\beta$-ball for estimation of $f(x_0).$ The rate is attained if the true function is bounded away from zero at $x_0.$ However, $n^{-\frac{\beta}{2\beta+1}}$ is not a minimax lower bound in the local sense of Theorem \ref{thm.lower_bound_sqrt}. Consider a sequence $(f_n^*)_n$ such that $f_n^*(x_0)\rightarrow 0$ as $n\rightarrow \infty.$ Then the local parameter space $\{f\in \mH^\beta: \KL(f, f_n^*)\leq 1\}$ also contains only functions vanishing at $x_0$ for large $n.$ Since we know from our upper bound in Theorem \ref{thm.est_link=sqrt} that for such functions the rate is faster than $n^{-\frac{\beta}{2\beta+1}}$, the latter can no longer be a lower bound. 

Obviously, one would like to restrict the parameter space to even smaller sets, by considering for instance shrinking Kullback-Leibler neighbourhoods around $f_n^*$. In this case however, the rates obtained for the upper and local lower bounds no longer match in general. To see this, consider the extreme case where the local parameter space consists of a single element. The rate obtained from the lower bound is zero, but there is of course no method that achieves zero risk over all parameters.

In order to see that flatness type conditions are necessary in order to achieve the rate of estimation $\tilde{r}_{n,\beta}^P(f(x))$ in \eqref{eq.local_rate_sqrt} for $\beta>1,$ we consider the class of non-negative linear functions on $[0,1].$ This class contains functions that have arbitrary H\"older smoothness but which are not flat near zero, for example $f(x)=x.$ For this class, the pointwise rate of estimation is bounded from below by $1/\sqrt{n\log n}$, as shown in the following result.

\begin{thm}\label{thm.Holder_counterexample}
Denote by $\mathcal{G}=\{f: f\geq 0, f(x)= ax+b \ \text{with} \  |a|+|b|\leq 2\}$ the class of non-negative linear functions on $[0,1]$ with bounded coefficients. Then there exists $C>0$ such that,
\begin{align*}
\liminf_{n \rightarrow \infty}  \inf_{\widehat f_n} \sup_{f \in \mathcal{G}: f(0) \leq n^{-1/2} } \P_f \left(\frac{|\hat{f}_n(0) - f(0)|}{(n\log n)^{-1/2}} \geq C \right) > 0 ,
\end{align*}
where the infimum is taken over all measurable estimators of $f(0)$.
\end{thm}

This should be compared to the definition of $\tilde{r}_{n,\beta}^P(f(x))$, which for $f(0) \leq n^{-1/2}$ and $\beta>1$ gives estimation rate  bounded by $(n^{-3/2}\log n)^{\beta/(2\beta+1)} \ll (n\log n)^{-1/2}.$ This example also shows that in model \eqref{eq.mod_link=sqrt} we do not get the fastest reconstruction for very smooth functions, such as linear functions, but we do for functions $f$ also satisfying flatness constraints in the sense that $f\in \mH^\beta$ with $\beta$ large.

\subsection{The Bernoulli case $h(x)= 2\arcsin \sqrt{x}$}
\label{sec.Bernoulli_case}

For the link function $h(x)= 2\arcsin \sqrt{x}$, model \eqref{eq.mod_link=h_direct} can be written
\begin{align}
	dY_t = 2 \arcsin (\sqrt{f(t)}) dt + n^{-1/2} dW_t, \quad t\in [0,1],
	\label{eq.mod_link=arcsin}
\end{align}
which is motivated by binary regression \cite{Grama1998} (see Section \ref{sec.specific_cases} for more details). We show that the rate of estimation for $f \in \mathcal{H}_B^\beta$ in \eqref{eq.mod_link=arcsin} at a point $x \in [0,1]$ is
\begin{align}
	\tilde{r}_{n,\beta}^B (f(x)) = \left( \frac{\log n}{n} \right)^{\frac{\beta}{\beta+1}} \vee  \left( f(x)(1-f(x))\frac{ \log n}{n} \right)^{\frac{\beta}{2\beta+1}}  
	\label{eq.local_rate_arcsin}
\end{align}
again giving two regimes depending on the size of $f(x)$. We could alternatively replace $f(x)(1-f(x))$ with $\min (f(x),1-f(x))$ above. Note that the first regime occurs if and only if $\min (f(x),1-f(x)) \lesssim (\log n/n)^{\beta/(\beta+1)}$, that is $f(x)$ is close to 0 or 1. We again see the effect of the non-linearity of $h$ on the rate since both 0 and 1 are irregular points (caused respectively by the $\sqrt{\cdot}$ and $\arcsin$ factors of $h$).

Given that \eqref{eq.mod_link=arcsin} arises from a binary response model \cite{Grama1998}, it is natural that the rate above behaves symmetrically in terms of $f$ and $1-f$. The rates are similar to those in the Poisson case \eqref{eq.local_rate_sqrt}, differing only in the second regime due to this symmetry. This means we again observe a superefficiency phenomenon near 0 and 1.
\begin{thm}\label{thm.est_link=arcsin}
For any $0 < \beta_1 \leq \beta_2 < \infty$ and $0< R_1 \leq R_2 < \infty$, the estimator $\widehat{f}$ satisfies
\begin{equation*}
\inf_{\substack{\beta_1 \leq \beta \leq \beta_2\\ R_1 \leq R \leq R_2}} \inf_{f :\|f\|_{\mH_B^\beta}\leq R} \P_f \left( \sup_{x \in [0,1]} \frac{|\widehat{f}(x) - f(x)|}{\tilde{r}_{n,\beta}^B(f(x))} \leq C(\beta_1,\beta_2,R_1,R_2) \right) \geq 1 - \frac{3n^{-(\tau^2/8-1)} }{\sqrt{2\log n}} ,
\end{equation*}
where $\tilde{r}_{n,\beta}^B(f(x))$ is given in \eqref{eq.local_rate_arcsin}.
\end{thm}
Again note that $\widehat{f}$ is fully adaptive, both spatially and over H\"older smoothness classes. We have a corresponding lower bound to the upper bound \eqref{eq.local_rate_arcsin},
\begin{align*}
	r_{n,\beta}^B(f(x)) = n^{-\frac{\beta}{\beta+1}} \vee ( f(x)(1-f(x))/n )^{\frac{\beta}{2\beta+1}} ,
\end{align*}
which again agrees with the upper bound $\tilde{r}_{n,\beta}^B(f(x))$ up to logarithmic factors.
\begin{thm}\label{thm.lower_bound_arcsin}
For any $\beta>0$, $R>0$, $x_0 \in [0,1]$ and any sequence $(f_n^*)_n$ with $\limsup_{n \rightarrow \infty}\|f_n^*\|_{\mH_B^\beta}<R$,
\begin{align*}
\liminf_{n \rightarrow \infty}  \inf_{\hat{f}_n(x_0)}  \sup_{\substack{f :\|f\|_{\mathcal{H}_B^\beta}\leq R \\ KL(f,f_n^*)\leq1 }}    \P_f \left(\frac{|\widehat f_n(x_0) - f(x_0)|}{r_{n,\beta}^B (f(x_0))} \geq C(\beta,R) \right) > 0 ,
\end{align*}
where the infimum is taken over all measurable estimators of $f(x_0)$.
\end{thm}

\subsection{The Gaussian variance case $h(x)= 2^{-1/2} \log x$}\label{sec.log_case}

For the link function $h(x)= 2^{-1/2} \log x$, model \eqref{eq.mod_link=h_direct} can be written
\begin{align}
	dY_t = \frac{1}{\sqrt{2}}\log (f(t)) dt + n^{-1/2} dW_t, \quad t\in [0,1],
	\label{eq.mod_link=log}
\end{align}
which is motivated by Gaussian variance estimation \cite{Grama1998} and spectral density estimation \cite{golubev2010} (see Section \ref{sec.specific_cases} for more details). We show that the rate of estimation for $f \in \mH^\beta$ in \eqref{eq.mod_link=log} at a point $x \in [0,1]$ is
\begin{align}
	\tilde{r}_{n,\beta}^G (f(x)) = 
	\begin{cases}
	 f(x)  &\text{if} \ \  f(x)\leq n^{-\beta}M_n, \\
	\Big( f(x)^2\frac{\log n}{n} \Big)^{\frac{\beta}{2\beta+1}} &\text{otherwise},
	\end{cases}
	\label{eq.local_rate_log}
\end{align}
where $M_n \rightarrow \infty$ grows strictly subpolynomially, that is, $M_n \ll n^{\delta}$ for any $\delta>0.$ For the exact expression for $M_n$, see the proof of Theorem \ref{thm.est_link=log} below.

The plug-in estimator $e^{\widehat{\sqrt{2} h}}$ is unable to adapt to small function values; small fluctuations in the function are magnified rendering estimation by wavelet thresholding unstable. However, this magnification also permits easier detection of this regime and we take advantage of this to propose a modified (non-adaptive) estimator. Define the local averaging process:
\begin{align*}
Z_n(x) = n \int_{x-\frac{1}{2n}}^{x+\frac{1}{2n}} dY_t = \frac{n}{\sqrt{2}} \int_{x-\frac{1}{2n}}^{x+\frac{1}{2n}} \log f(t) dt + \epsilon_n (x),\quad x\in\left[ \frac{1}{2n}, 1-\frac{1}{2n} \right],
\end{align*}
where $\epsilon_n (x) = \sqrt{n}(W_{x+\frac{1}{2n}} - W_{x-\frac{1}{2n}}) \sim N(0,1)$. For the boundary values, we truncate the integral so that for $x\in[0,1/(2n)]$,
\begin{align*}
Z_n(x) = \frac{1}{x+1/(2n)} \int_0^{x+\frac{1}{2n}} dY_t,
\end{align*}
with the corresponding upper truncation for $x\in[1-1/(2n),1]$ and $\epsilon_n(x)$ extended to these ranges. The process $Z_n$ is used to detect which regime of $\tilde{r}_{n,\beta}^G (f(x))$ occurs, since then the zero estimator and plug-in estimator are optimal on the first and second regimes respectively. This yields the estimator 
\begin{align*}
\widehat{f}(x) = e^{\sqrt{2}\widehat{h}(x)}  \mathbf{1} \big( x :Z_n(x) \geq - 2^{-1/2}\beta\log n + \sigma \sqrt{2\log n} \big)  ,
\end{align*}
where $\widehat{h}$ is as in \eqref{eq.hard_wav_est} and $\sigma > 1$.
\begin{thm}\label{thm.est_link=log}
For any $\beta>0$ and $0< R_1 \leq R_2 < \infty$, the estimator $\widehat{f}$ satisfies
\begin{equation*}
\inf_{R_1 \leq R \leq R_2}  \inf_{f :\|f\|_{\mH^\beta}\leq R}
\P_f \left( \sup_{x \in [0,1]} \frac{|\widehat{f}(x) - f(x)|}{\tilde{r}_{n,\beta}^G(f(x))} \leq C(\beta,R_1,R_2) \right) \geq 1 -\frac{3n^{-(\tau^2/8-1)} }{\sqrt{2\log n}} - C(\sigma) n^{-\tfrac{\sigma-1}{2}} ,
\end{equation*}
where $\tilde{r}_{n,\beta}^G(f(x))$ is given in \eqref{eq.local_rate_log}.
\end{thm}
We note that $\widehat{f}$ is spatially adaptive, but does not adapt over H\"older classes. The difficulty in adapting to smoothness compared to the Poisson and Bernoulli models is due to the nature of the irregular point 0. In the other models, the singularity reduces the smoothness inherited from $f$, but still maps small values to small values. On the contrary, $h(x) = 2^{-1/2} \log x$ is unbounded near 0, which causes the difference in the rate for small values. This highlights how the nature of the non-linearity affects estimation in the model.

The corresponding lower bound is
\begin{align*}
	r_{n,\beta}^G(f(x)) = f(x) \wedge ( f(x)^2/n )^{\frac{\beta}{2\beta+1}},
\end{align*}
which agrees with the upper bound $\tilde{r}_{n,\beta}^G(f(x))$ in \eqref{eq.local_rate_log} up to subpolynomial factors.
\begin{thm}\label{thm.lower_bound_log}
For any $\beta>0$, $R>0$, $x_0 \in [0,1]$ and any sequence $(f_n^*)_n$ with $\limsup_{n \rightarrow \infty}\|f_n^*\|_{\mH^\beta}<R$,
\begin{align*}
\liminf_{n \rightarrow \infty}  \inf_{\hat{f}_n(x_0)} \sup_{\substack{f:\|f\|_{\mH^\beta}\leq R \\ KL(f,f_n^*)\leq1 }}    \P_f \left(\frac{|\hat{f}_n(x_0) - f(x_0)|}{r_{n,\beta}^G (f(x_0))} \geq C(\beta,R) \right) > 0 ,
\end{align*}
where the infimum is taken over all measurable estimators of $f(x_0)$.
\end{thm}
Since $h(x) = 2^{-1/2} \log x$ magnifies changes near 0, the first regime in $r_{n,\beta}^G(f(x))$ depends much more closely on the true function value than in the Poisson or Binomial cases. In those cases one can always construct an alternative that has pointwise distance $n^{-\beta/(\beta+1)}$ from $f_n^*$ and has Kullback-Leibler distance 1 from $f_n^*$. In contrast here, any alternative that is more than a constant multiple of $f_n^*(x_0)$ away from $f_n^*$ at $x_0$ will have diverging Kullback-Leibler distance from $f_n^*$. This is due to the extremely sensitive irregular point of $h$.

\section{Extension to inverse problems}
\label{sec.extension}

In the previous section, we derived local rates for $f$ in the direct case. For the specific link functions considered, we therefore have explicit and rate optimal bounds $\widetilde r_{n,\beta}(f(x))$ such that $|\widehat{f}(x) - f(x)|\leq \widetilde r_{n,\beta}(f(x))$ for all $x\in [0,1]$ with high probability. Except for the logarithmic link function, the rates are adaptive, that is, the smoothness index of $f$ is not necessarily assumed known. 

Recall the full inverse problem \eqref{eq.mod_link=h}. In this case the drift function is $h\circ Kf$ and by applying the wavelet thresholding procedure described above, we obtain an estimator $\widehat{Kf}$ for $Kf,$ which, as explained in the introduction, can be thought of as a data pre-smoothing procedure. Viewing $Y^\delta(t) := \widehat{Kf}(t)$ as the new observation in a deterministic inverse problem, we can control the local noise level $\delta$ in the sense that
\begin{align*}
	|Y^\delta (t) - Kf(t) | \leq \delta(t):=\widetilde r_{n,\beta}(Kf(t)), \quad \quad \text{for all}  \ t\in [0,1],
	\end{align*}
with high probability. In order to solve the deterministic inverse problem, reconstructing the function $f$ from the pre-smoothed data $(Y^{\delta}(t))_{t\in [0,1]},$ we need to be able to compute the order of the noise level. This can then be used for the choice of the regularization parameter or stopping criteria in the regularization procedure. The next result shows that with high probability, a plug-in estimator for $Kf$ does not change the order.

\begin{thm}
\label{thm.plug_in_bd}
On the event $|Y^\delta(t)-Kf(t)|\leq C\widetilde r_{n,\beta}^X (Kf(t)),$ where $X \in\{P, B, G\}$ is one of the three scenarios considered in Sections \ref{sec.Poisson_case}-\ref{sec.log_case}, there is a constant $\overline C$ that is independent of $t,$ such that 
\begin{align*}
	\overline C^{-1} \widetilde r_{n,\beta}^X(Y^\delta(t)) 
	&\leq \widetilde r_{n,\beta}^X(Kf(t)) \leq \overline C\widetilde r_{n,\beta}^X(Y^\delta(t)), \quad X\in \{P,B\}
\end{align*}
and
\begin{align*}
	\overline C^{-1} \widetilde r_{n,\beta}^G(Y^\delta(t)) 
	&\leq \widetilde r_{n,\beta}^G(Kf(t)  ) \leq \overline C\widetilde r_{n,\beta}^G(Y^\delta(t)) + n^{-\beta} M_n,
\end{align*}
with $M_n$ as in \eqref{eq.local_rate_log}.
\end{thm}

We can therefore estimate the order of the noise level using the plug-in estimator $\widetilde r_{n,\beta}(Y^\delta(t))$ since by Theorems \ref{thm.est_link=sqrt}, \ref{thm.est_link=arcsin}, and \ref{thm.est_link=log}, $|Y^\delta(t)-Kf(t)|\leq C\widetilde r_{n,\beta} (Kf(t))$ holds with high probability. 

The order of the noise level still depends on the smoothness index $\beta$ of $Kf,$ which might be unknown. At the moment there is no completely satisfactory answer to this problem and we briefly outline a practical approach below. If the operator $K$ is $s$-times integration, then $K: \mH^\alpha \rightarrow \mH^{\alpha+s}$ for any $\alpha>0,$ as proved in Theorem 4.1 of \cite{RaySchmidt-Hieber2015c}. If $f$ has known smoothness $\alpha$, we can then conclude that $\beta=\alpha+s.$ For convolution type operators, we believe that similar results hold in view of the Fourier multiplier theorems on Besov spaces (and thus also for H\"older spaces) in \cite{girardi2003}, Section 4 and in \cite{arendt2004}.

From a practical point of view, we can guess the smoothness of $Kf$ by studying how many empirical wavelet coefficients are kept in the series estimator \eqref{eq.hard_wav_est}. For smoother signals, the wavelet coefficients decay more quickly and so fewer coefficients are taken into account by the estimator. Estimators based on such ideas typically work in practice but have poor properties in the minimax sense over all H\"older functions, since these function spaces contain rough functions that look much smoother at many resolution levels.

\section{Details and examples for model \eqref{eq.mod_link=h}}
\label{sec.specific_cases}

To motivate model \eqref{eq.mod_link=h}, let us first consider a regular parametric problem where we observe $n$ i.i.d. copies of a random variable with distribution $P_\theta$, where $\theta\in \Theta \subset \R$ is the parameter of interest. The Maximum Likelihood Estimator (MLE) $\widehat{\theta}^{\MLE}$ is under weak assumptions efficient and asymptotically normal in the sense that $\widehat{\theta}^{\MLE}= \mathcal{N}(\theta, (nI(\theta))^{-1})+o_P(n^{-1/2}),$ where $I(\theta)$ denotes the Fisher information and $\mathcal{N}(\mu,\sigma^2)$ is a normal random variable with mean $\mu$ and variance $\sigma^2.$ If we can find a function $h:\R \rightarrow \R$ satisfying
\begin{align}
	h'(\theta) = \sqrt{I(\theta)},
	\label{eq.VST}
\end{align}
then, using the delta method (e.g. Chapter 3 of \cite{vanderVaart1998}), the transformed MLE satisfies $h(\widehat{\theta}^{\MLE})= \mathcal{N}(h(\theta), 1/n)+o_P(n^{-1/2}).$ The remainder term does not carry any information about $\theta$ that is relevant for asymptotic statements and we have thus rewritten the original problem as a Gaussian model, where we observe 
\begin{align}
 Y=h(\theta)+n^{-1/2}\epsilon \quad \text{with} \ \ \epsilon\sim \mathcal{N}(0,1).
 \label{eq.param_trafo}
 \end{align} 
Since the Fisher information is positive for regular problems, $h$ is a strictly monotone increasing function. By rewriting the model in the way just described, we do not lose information as the sample size tends to infinity and we may thus view the transformed model \eqref{eq.param_trafo} as an equivalent representation of the original problem. 

This concept can be extended to nonparametric problems but the situation becomes more subtle, since the MLE does not in general exist. A very strong notion of approximation of models is convergence in Le Cam distance (cf. \cite{Grama1998, nussbaum1996}). Suppose we observe independent random variables $(X_i)_{i=1,\ldots,n}$ following a distribution from an exponential family with real-valued parameters $\theta_i = f(\tfrac in)$, where $f :[0,1]\rightarrow \R$ is an unknown regression function to be estimated. Grama and Nussbaum \cite{Grama1998} showed that under some minimal smoothness assumptions, this model converges in Le Cam distance to the Gaussian model where we observe the path $(Y_t)_{t\in [0,1]}$ with
\begin{align*}
	dY_t = (h \circ f)(t) dt + n^{-\frac 12} dW_t, \quad t\in [0,1],
\end{align*}
where the link function $h:\R \rightarrow \R$  satisfies \eqref{eq.VST}. This choice of $h$ is related to the \textit{variance-stabilizing transformation} for the exponential family (see e.g. \cite{Brown2010} for more details). Such a result may be viewed as a nonparametric analogue of the parametric approximation explained above. The approximation becomes better for smoother functions $f$ and deteriorates as we approach irregular points of $h$, where $h'$ tends to infinity. For the specific case of Poisson intensity estimation, it has been shown in \cite{RaySchmidt-Hieber2015b} that the convergence with respect to Le Cam distance still holds true near the irregular point under certain assumptions. 

We now extend this to the full notion of inverse problems. Suppose we observe independent random variables $(X_i)_{i=1,\ldots,n}$ following a distribution from an exponential family with parameters $\theta_i = (Kf) (\tfrac in),$ with $K$ a linear (ill-posed) operator mapping univariate functions to univariate functions and $f$ is unknown. Using the same transformation as in the direct case, we obtain approximating model \eqref{eq.mod_link=h}, where we observe $(Y_t)_{t\in [0,1]}$ with
\begin{align*}
	dY_t = (h \circ Kf)(t) dt + n^{-\frac 12} dW_t, \quad t\in [0,1].
\end{align*}
The quality of approximation can even be better than in the direct case, since $Kf$ is typically smoother than $f.$ However, there are other obstacles in the formulation which make formal convergence proofs with respect to the Le Cam distance difficult. Nevertheless, it is believed that from a practical point of view the approximation is sufficiently accurate.

The above approach has been outlined for nonparametric exponential families with fixed uniform design but is in fact more general. For instance, it is well-known that nonparametric density estimation, where we observe $n$ i.i.d. copies of a random variable $X$ drawn from a Lebesgue density $f$, can be mapped to the model \eqref{eq.mod_link=sqrt} (cf. \cite{nussbaum1996}, \cite{brown2004}, \cite{RaySchmidt-Hieber2015b}), which is model \eqref{eq.mod_link=h} with link function $h(x) = 2\sqrt{x}.$ Following the above arguments, we may extend this to density deconvolution, which is one of the most studied statistical inverse problems. Here, we observe $n$ i.i.d. copies of $X+\epsilon,$ where $X$ and $\epsilon$ have Lebesgue densities $f$ and $g$ respectively with $g$ known, and we aim to reconstruct $f$ from the data. The density of $X+\epsilon$ is then the convolution $f\star g$ and we may thus rewrite deconvolution as a Gaussian shift model, where we observe $(Y_t)_{t\in [0,1]}$ with
\begin{align*}
	dY_t = 2\sqrt{(f\star g)(t)} dt + n^{-1/2} dW_t, \quad t\in [0,1].
\end{align*}

Another important example is Poisson intensity estimation. Suppose we observe a Poisson process on $[0,1]$ with intensity $nf,$ where $f:[0,1]\rightarrow (0,\infty).$ This can be replaced by model \eqref{eq.mod_link=sqrt}, so that Poisson intensity estimation gives the same link function as density estimation. We may also consider the related inverse problem where we observe a Poisson process on $[0,1]$ with intensity $nKf,$ where $K$ is a linear operator mapping into the space of univariate functions on $[0,1].$  The $2$-dimensional version of this problem has various applications in photonic imaging. In this case, $K$ is typically a convolution operator modelling the blurring of images by a so-called point spread function. In analogy with the density deconvolution case, we may rewrite this as observing  $(Y_t)_{t\in [0,1]}$ with
\begin{align*}
	dY_t = 2\sqrt{(Kf)(t)} dt + n^{-1/2} dW_t, \quad t\in [0,1].
\end{align*}

Another interesting example is binary regression, where we observe $n$ independent Bernoulli random variables (coin flips) with probabilities of success $f(\tfrac in ) = P(X_i=1)$, where $f:[0,1] \rightarrow [0,1]$ is an unknown regression function \cite{diaconis1993}. The Bernoulli distribution is an exponential family and we can use the approach outlined above to rewrite this in the form \eqref{eq.mod_link=arcsin} (cf. \cite{Grama1998}, Section 4). The link function in this example is $h(x) = 2\arcsin \sqrt{x}.$ In the inverse problem case, we may still use the same approximation. The probability of success of $X_i$ is then $(Kf)(\tfrac in)$, where $K$ is an operator that maps into the space of univariate functions taking values in $(0,1).$

A further example is variance estimation. Suppose that in the original problem, we observe $n$ independent normal random variables with variance $f(\tfrac in)^2,$ where $f\geq 0$ is an unknown regression function. This problem can be rephrased as a nonparametric estimation problem in terms of exponential families and leads to model \eqref{eq.mod_link=log} (cf. \cite{Grama1998}, Section 4). We may also extend this to an ill-posed inverse problem, where we observe $X_i \sim \mathcal{N}(0, (Kf)(\tfrac in)^2),$ $i=1,\ldots,n,$ and $K$ could for instance be a convolution operator. The accompanying shift model is then \eqref{eq.mod_link=log} with $f$ replaced by $Kf.$

Finally, let us consider spectral density estimation, where we observe a random vector of length $n$ coming from a stationary Gaussian distribution with spectral density $f:[-\pi, \pi]\rightarrow (0,\infty).$ In the corresponding Gaussian shift model, we observe $(Y_t)_{t\in [-\pi,\pi]}$ with
\begin{align*}
	dY_t = \frac{1}{2\sqrt{\pi}}\log (f(t)) dt + n^{-1/2} dW_t, \quad t\in [-\pi,\pi],
\end{align*}
(cf. \cite{golubev2010}). Notice that this model is of the same form as for variance estimation.

To summarise, various important statistical inverse problems can be approximated by model \eqref{eq.mod_link=h} and often even some theoretical guarantees exist that asymptotically the approximation does not lead to a loss of information. For most applications, the link function $h$ is non-linear. Bringing models into the form \eqref{eq.mod_link=h} allows one to develop a unified estimation theory for such statistical inverse problems.

\section*{Acknowledgements} We thank Aad van der Vaart and two anonymous referees for many interesting remarks and suggestions.

\appendix

\section*{Appendix}

The appendix is subdivided into Section \ref{sec.proofs}, which contains the proofs of the main theorems, and Section \ref{sec.technical_results}, where we collect some technical results. Recall that the estimator of $h\circ f$ is denoted for convenience by $\widehat{h}.$

\section{Proofs}
\label{sec.proofs}

\begin{proof}[Proof of Theorem \ref{thm.est_link=sqrt}]
Let $h(f) = \sum_{j,k} d_{j,k} \psi_{j,k}$ (here $h(f) = 2\sqrt{f}$) and for $\gamma > 0$ define
\begin{equation*}
\mathcal{J}_n (\gamma) = \big\{ (j,k) : |d_{j,k}| > \gamma \sqrt{\log n / n} \big\}  .
\end{equation*}
For $\tau > 2\sqrt{2}$ and every $0 < \gamma < \tau$ it follows that
\begin{align*}
\P \big( |\widehat{d}_{j,k}| \neq 0 \text{ for some } (j,k) \in \mathcal{J}_n^c (\gamma) \big) 
&\leq \frac{2n^{ 1-\frac{1}{2} (\tau - \gamma)^2}}{(\tau-\gamma)\sqrt{\log n}} , \\
\P \big( |\widehat{d}_{j,k}| = 0 \text{ for some } (j,k) \in \mathcal{J}_n (2\tau) \big) 
&\leq \frac{2n^{1-\frac{\tau^2}{2}} }{\tau \sqrt{\log n}} ,
\end{align*}
\begin{equation*}
\P \left( |Y_{j,k} - d_{j,k}| > \tau \sqrt{\log n / n} \text{ for some } (j,k), 0 \leq j \leq J_n, 0 \leq k \leq 2^j-1  \right) \leq \frac{2n^{1-\frac{\tau^2}{2}} }{\tau \sqrt{\log n}} .
\end{equation*}
The first and third inequalities follow from the one-dimensional Gaussian tail-bound $\P (|Z| \geq t) \leq 2(2\pi)^{-1/2} t^{-1} e^{-t^2/2}$ and a union bound over the $2^{J_n+1} \leq 2n$ possible indices, while the second event is contained in the third event. Consider the event
\begin{equation}\label{eq.wavelet_event}
\begin{split}
A_n & = \Big\{ |\widehat{d}_{j,k}| = 0,  \forall (j,k) \in \mathcal{J}_n^c (\tau/2) \Big\} \cap \Big\{ |\widehat{d}_{j,k}| \neq 0,  \forall (j,k) \in \mathcal{J}_n (2\tau)   \Big\}  \\
& \quad \quad \quad  \cap \Big\{ |Y_{j,k} - d_{j,k}| \leq \tau \sqrt{\log n/n},  \forall (j,k), 0 \leq j \leq J_n \Big\}  ,
\end{split}
\end{equation}
which by the three inequalities above satisfies $\P (A_n^c) \leq 6 \tau^{-1} (\log n)^{-1/2} n^{1-\frac{\tau^2}{8}}$.

We bound the loss at a given point $x_0 \in (0,1)$, noting that we can do this simultaneously for all such $x_0$ on the event $A_n$. Recalling that $h^{-1}(x) = x^2/4$ and so $\widehat{f} = h^{-1} (\widehat{h}) = \widehat{h}^2/4$,
\begin{align*}
	|\widehat{f}(x_0)- f(x_0)|& = | (\widehat{h}(x_0) - 2\sqrt{f(x_0)} + 2\sqrt{f(x_0)})^2/4- f(x_0) |  \\
	& \leq \sqrt{f(x_0)} \big| \sum_{j,k} (\widehat{d}_{j,k}-d_{j,k}) \psi_{j,k}(x_0) \big| 
	+ \frac{1}{4}\big| \sum_{j,k} (\widehat{d}_{j,k}-d_{j,k}) \psi_{j,k}(x_0) \big| ^2 = : (I) +(II) .
\end{align*}
We bound terms $(I)$ and $(II)$ on the event $A_n$. By the localization property of wavelets,
\begin{align}
	\big| \sum_{j,k} (\widehat{d}_{j,k}-d_{j,k}) \psi_{j,k}(x_0) \big| 
	\leq C(\psi) \sum_{j} 2^{j/2} \max_{k:\, \psi_{j,k}(x_0) \neq 0} |\widehat{d}_{j,k} - d_{j,k} |. 
	\label{eq.Besov_embed}
\end{align}
We require bounds on the wavelet coefficients $| \langle 2 \sqrt{f} ,\psi_{j,k} \rangle|$. Let $j(x_0)$ be as in Lemma \ref{lem.wav_decay_small}, setting $j(x_0) = \infty$ if $f(x_0)=0$. Lemma \ref{lem.wav_decay_small} yields that for any $\psi_{j,k}$ with $j\geq j(x_0)$ and $x_0\in \supp(\psi_{j,k})$,
\begin{equation*}
|d_{j,k}| = \Big| \int_0^1 2\sqrt{f(t) }\psi_{j,k}(t) dt \Big| \leq C(\psi,\beta)\frac{R}{\sqrt{f(x_0)}} 2^{-\frac{j}{2}(2\beta+1)}.
\end{equation*}
Conversely, for the low-frequency coefficients $j<j(x_0)$ with $\psi_{j,k}(x_0) \neq 0$, Lemma \ref{lem.wav_decay_small} yields $|d_{j,k}| \leq C(\psi,\beta) \sqrt{R} \, 2^{-\frac{j}{2}(\beta+1)}$.

Recall that on $A_n$, we have $\widehat{d}_{j,k} = 0$ for all $(j,k) \in \mathcal{J}_n^c (\tau/2)$. Denote by $J_1$ the smallest integer $j$ for which $C\sqrt{R}2^{-\frac j2 (\beta+1)}\leq \frac{\tau}{2} \sqrt{\log n/n}$ and similarly denote by $J_2$ the smallest $j$ for which $CR f(x_0)^{-1/2} 2^{-\frac j2 (2\beta+1)}\leq \frac{\tau}{2} \sqrt{\log n/n}.$ By simple calculations we find that $2^{J_1} \asymp (n/\log n)^{1/(\beta+1)}$ and $2^{J_2} \asymp (n/\log n)^{1/(2\beta+1)} f(x_0)^{-1/(2\beta+1)}.$ Notice that $2^{J_1} \lesssim 2^{J_2}$ if and only if $f(x_0) \lesssim (\log n/ n)^{\beta / (\beta+1)}$, that is, the transition in the rate $\tilde{r}_{n,\beta}^P(f(x))$.

We now bound the sum \eqref{eq.Besov_embed} restricted to all $(j,k) \in \mathcal{J}_n^c (\tau/2)$. If $J_1\leq J_2$ (``small" $f(x_0)$), then on $A_n$,
\begin{align*}
	\sum_{(j,k) \in \mathcal{J}_n^c (\tau/2)} 2^{j/2} |d_{j,k}| \leq C(\beta,R,\tau) \Big[ \sum_{j\leq J_1} 2^{j/2}  \sqrt{\frac{\log n}{n}} + \sum_{J_1<j} 2^{j/2} 2^{-\frac j2(\beta+1)} \Big]
	\leq C \left( \frac{\log n}{n}\right)^{\frac{\beta}{2\beta+2}}.
\end{align*}
If $J_2 \leq J_1$ (``large" $f(x_0)$) we have similarly
\begin{align*}
	\sum_{(j,k) \in \mathcal{J}_n^c (\tau/2)} 2^{j/2} |d_{j,k}| & \leq C(\beta,R,\tau) \Big[ \sum_{j\leq J_2} 2^{j/2}  \sqrt{\frac{\log n}{n}} +  \sum_{J_2<j}  2^{j/2} f(x_0)^{-1/2} 2^{-\frac j2(2\beta+1)} \Big] \\
	& \lesssim 2^{J_2/2}  \sqrt{\frac{\log n}{n}} + f(x_0)^{-1/2} 2^{-J_2\beta}
	\lesssim f(x_0)^{-1/(4\beta+2)} \left(\frac{\log n}{n}\right)^{\frac{\beta}{2\beta+1}}.
\end{align*}
For the remaining coefficients $(j,k) \in \mathcal{J}_n (\tau/2)$, note that on $A_n$ we have $|\widehat{d}_{j,k} - d_{j,k}| \leq 2\tau \sqrt{\log n/n}$. Consequently, the sum \eqref{eq.Besov_embed} restricted to these indices is bounded by $C (2^{J_1/2} \wedge 2^{J_2/2}) \sqrt{\log n/n}$, which is bounded from above by the minimum of the previous two displays.

Together, the bounds from the bias and stochastic error show that \eqref{eq.Besov_embed} can be bounded from above by a multiple of $\min\{(\log n/n)^{\beta/(2\beta+2)} , f(x_0)^{-1/(4\beta+2)} (\log n/n)^{\beta/(2\beta+1)}\}.$ From this we finally deduce the upper bounds for $(I)$ and $(II)$,
\begin{align*}
		(I) \lesssim \min \Big ( \sqrt{f(x_0)} (\log n/n)^{\beta/(2\beta+2)} , f(x_0)^{\beta/(2\beta+1)} (\log n/n)^{\beta/(2\beta+1)} \Big),
\end{align*}
\begin{align*}
	(II) \lesssim \min \Big ( (\log n/n)^{\beta/(\beta+1)} , f(x_0)^{-1/(2\beta+1)} (\log n/n)^{2\beta/(2\beta+1)} \Big).
\end{align*}
In both the bounds for $(I)$ and $(II),$ the minimum is attained by the first value if and only if $f(x_0)\leq (\log n/n)^{\beta/(\beta+1)}.$ Comparing the various terms gives $$(I)+(II)\lesssim (\log n/n)^{\beta/(\beta+1)} + f(x_0)^{\beta/(2\beta+1)} (\log n/n)^{\beta/(2\beta+1)} \lesssim \tilde{r}_{n,\beta}^P (f(x_0)).$$
\end{proof}

\begin{proof}[Proof of Theorem \ref{thm.lower_bound_sqrt}]
We distinguish between the two cases
\begin{align*}
	(A): \quad f_n^*(x_0) > n^{-\frac{\beta}{\beta+1}} , \quad \quad \quad \quad (B): \quad f_n^*(x_0) \leq n^{-\frac{\beta}{\beta+1}} ,
\end{align*}
which correspond to the two regimes in $r_{n,\beta}^P(f(x_0))$. In both cases we derive the lower bound using a two hypothesis testing argument as described in Section 2 of \cite{tsybakov2009}. Notice that the original lower bound reduction scheme does not directly permit rates that depend on the parameter itself. However, this can easily be fixed by assuming that for the two hypotheses, the corresponding rates are of the same order. We need to verify that
\begin{itemize}
\item[$(i)$] there exist two sequences of hypotheses $(f_{0,n})_n$ and $(f_{1,n})_n$ in the local parameter space $\{ f : \|f\|_{\mH^\beta}\leq R, \KL(f,f_n^*)\leq 1\}$,
\item[$(ii)$] $|f_{0,n}(x_0)-f_{1,n}(x_0)| \gtrsim r_{n,\beta}^P (f_{0,n}(x_0))\asymp r_{n,\beta}^P (f_{1,n}(x_0)),$
\item[$(iii)$] the Kullback-Leibler distance between the hypotheses is $O(1)$ (which in the present case follows from $(i)$).
\end{itemize}
If $(i)$-$(iii)$ are satisfied then applying Theorem 2.2(iii) of \cite{tsybakov2009} completes the proof.

$(A)$: Consider the hypotheses
\begin{align}
f_{0,n}(x) := f_n^*(x) , \quad  \ f_{1,n}(x) := f_n^*(x)+ R h_n^\beta K\big(\tfrac{x-x_0}{h_n} \big) , \ \ h_n = c_0 \left(\tfrac{f_n^*(x_0)}{n}\right)^{1/(2\beta+1)},
\label{eq.two_hyp_large}
\end{align}
where $c_0>0$ and $K = \eta K_0$ for $K_0$ the function in Lemma \ref{lem.mollifier} with $\eta>0$. Recall that $K_0$ is symmetric, non-negative, infinitely-differentiable and supported on $[-1,1]$. For $(i)$, it holds by assumption that $\limsup_{n}\|f_n^*\|_{\mH^\beta} < R$. Taking the $\lfloor \beta \rfloor$-th derivative yields $f_{1,n}^{(\lfloor \beta \rfloor )}(x) = f_{0,n}^{(\lfloor \beta \rfloor )}(x)+ R \eta h_n^{\beta - \lfloor \beta \rfloor } K_0^{(\lfloor \beta \rfloor )} ( (x-x_0)/h_n)$ so that
\begin{align*}
\big\| R h_n^\beta K\big(\tfrac{x-x_0}{h_n} \big) \big\|_{C^\beta} \leq R\eta(h_n^\beta \|K_0\|_\infty+ h_n^{\beta - \lfloor \beta \rfloor} \|K_0^{(\lfloor \beta \rfloor)} \|_\infty + h_n^{\beta - \lfloor \beta \rfloor} C(K_0) ),
\end{align*}
where $C(K_0) = \sup_{x\neq y}|K_0^{(\lfloor\beta\rfloor)}(x) - K_0^{(\lfloor\beta\rfloor)}(y)|/|x-y|^{\beta-\lfloor\beta\rfloor}<\infty$. Taking $\eta>0$ small enough, the above is bounded by $(R-\limsup_n \|f_n^*\|_{\mH^\beta})/2>0$. For $j = 1,...,\lfloor\beta\rfloor$, using Lemma \ref{lem.mollifier}
\begin{align*}
\Big| \big( R h_n^\beta K\big(\tfrac{x-x_0}{h_n} \big) \big)^{(j)} \Big| \leq C(\beta,j,K_0) (R \eta)^\frac{j}{\beta} \Big| Rh_n^\beta K \big(\tfrac{x-x_0}{h_n} \big) \Big|^\frac{\beta-j}{\beta},
\end{align*}
so that again taking $\eta>0$ small enough, we have $|R h_n^\beta K((\cdot-x_0)/h_n)|_{\mH^\beta} \leq (R-\limsup_n \|f_n^*\|_{\mH^\beta})/2$. Using the definition of $\|\cdot\|_{\mH^\beta}$ and that it defines a norm by Theorem 2.1 of \cite{RaySchmidt-Hieber2015c},
\begin{align*}
\|f_{1,n}\|_{\mH^\beta} \leq \|f_{0,n}\|_{\mH^\beta} + \big\| R h_n^\beta K\big(\tfrac{x-x_0}{h_n} \big) \big\|_{C^\beta} + \big| R h_n^\beta K\big(\tfrac{x-x_0}{h_n} \big) \big|_{\mH^\beta}  \leq R.
\end{align*}
This verifies that $\|f_{0,n}\|_{\mH^\beta}, \|f_{1,n}\|_{\mH^\beta}\leq R$.

Next, the pointwise distance satisfies $|f_{0,n}(x_0)- f_{1,n}(x_0)| = R c_0^\beta K(0) (f_n^*(x_0) /n)^{\frac{\beta}{2\beta+1}} \gtrsim r_{n,\beta}^P(f_{0,n}(x_0))$ since $f_n^*(x_0) > n^{-\frac{\beta}{\beta+1}}$ by assumption. Since $h_n^\beta = c_0^\beta (f_n^*(x_0)/n)^{\beta/(2\beta+1)} \lesssim f_n^*(x_0),$ we have that $r_{n,\beta}^P(f_{0,n}(x_0)) \asymp  r_{n,\beta}^P (f_{1,n}(x_0))$, thereby establishing $(ii)$.

Finally, we bound the Kullback-Leibler divergence between these hypotheses. Applying Lemma \ref{lem.local_func_size},
\begin{align*}
	KL (P_{f_{0,n}}, P_{f_{1,n}}) 
	& = \frac{n}{2} \int_0^1 \big( 2\sqrt{f_{0,n}(x)} - 2\sqrt{f_{1,n}(x)} \big)^2 dx \\
	& = 2nh_n \int_{-1}^1 \left( \sqrt{f_n^*(x_0 + h_nu) } - \sqrt{f_n^*(x_0 + h_n u) + Rh_n^\beta K(u))}\right)^2 du \\
	& \leq \frac{4nh_n}{f_n^*(x_0)} \int_{-1}^1 \left( f_n^*(x_0 + h_nu) - f_n^*(x_0 + h_n u) - Rh_n^\beta K(u) )\right)^2 du \\
	& = 4R^2 c_0^{2\beta+1} \| K \|_2^2 \leq 1
\end{align*}
for $c_0 \leq aR^{-1/\beta} \wedge (4R^2\|K\|_2^2)^{-1/(2\beta+1)}$ and $a$ as in Lemma \ref{lem.local_func_size}. This verifies $(iii)$ and, when combined with the above, $(i)$.

$(B)$: Consider the hypotheses
\begin{align}
f_{0,n}(x) = f_n^*(x), \quad \quad \quad  \ f_{1,n}(x) = f_n^*(x) + R h_n^\beta K\big(\tfrac{x-x_0}{h_n} \big) , \ \ h_n = c_0 n^{-1/(\beta+1)},
\label{eq.two_hyp_small}
\end{align}
where $c_0 >0$ and $K$ is as in $(A)$. Note that the bandwidth $h_n$ is different in this case. Exactly as in $(A)$, it can be shown that $f_{0,n}$ and $f_{1,n}$ are in the parameter space, with the only difference being the bound on the Kullback-Leibler divergence. Using that the square-root function is $1/2$-H\"older continuous with H\"older constant 1,
\begin{align*}
	KL (P_{f_{0,n}}, P_{f_{1,n}}) & = 2n \int \big( \sqrt{f_{0,n}(x)} - \sqrt{f_{1,n}(x)} \big)^2 dx \\
	& \leq 2n \int |f_{0,n}(x) - f_{1,n}(x)| dx = 2nRh_n^{\beta+1} \| K \|_1 = 2Rc_0^{\beta+1} \|K \|_1 \leq 1,
\end{align*}
for $c_0 \leq (2R\|K\|_1)^{-1/(\beta+1)}$. This verifies $(i)$ and $(iii)$. Finally for $(ii)$ we have $|f_{0,n}(x_0) - f_{1,n}(x_0)| = R c_0^{\beta} K(0)n^{-\beta/(\beta+1)},$ which completes the proof.
\end{proof}

\begin{proof}[Proof of Theorem \ref{thm.Holder_counterexample}]
Consider the two hypotheses
\begin{align*}
	f_{0,n}(x) = x, \quad \quad \quad f_{1,n}(x)= x+r_n, \ \ r_n = c_0 /\sqrt{n\log n} ,
\end{align*}
where $c_0 >0$. We verify $(i)$-$(iii)$ as in the proof of Theorem \ref{thm.lower_bound_sqrt}. Both $(i)$ and $(ii)$ hold since the functions $f_{0,n}$ and $f_{1,n}$ are in the class $\mathcal{G},$ upper bounded at zero by $n^{-1/2}$ and satisfy $|f_{0,n}(0)- f_{1,n} (0)| = r_n.$ For the Kullback-Leibler divergence,
\begin{align*}
KL (P_{f_{0,n}}, P_{f_{1,n}}) 
	&= \frac{n}{2} \int_0^1 \big(2\sqrt{x+r_n}-2\sqrt{x}\big)^2 dx \\
	&\leq  4n \int_0^{r_n} (\sqrt{x+r_n})^2 +(\sqrt{x})^2 dx + 2n \int_{r_n}^1 \frac{(x+r_n-x)^2}{(\sqrt{x+r_n}+ \sqrt{x})^2} dx \\
	&\leq 8n r_n^2 + 2nr_n^2 \int _{r_n}^1 \frac 1x dx \\
	&\leq  10nr_n^2 \log(1/r_n),
\end{align*}
which is smaller than $1$ for $c_0$ small enough, thereby verifying $(iii)$.
\end{proof}

\begin{proof}[Proof of Theorem \ref{thm.est_link=arcsin}]
We follow a similar approach as in Theorem \ref{thm.est_link=sqrt}. Again let $h(f) = \sum_{j,k} d_{j,k} \psi_{j,k}$ (here $h(f) = 2\arcsin \sqrt{f}$) and let $A_n$ be the event in \eqref{eq.wavelet_event}. For fixed $x_0 \in (0,1)$ we again bound \eqref{eq.Besov_embed}, but with $(d_{j,k})$, $(\widehat{d}_{j,k})$ corresponding to model \eqref{eq.mod_link=arcsin}. We require the corresponding bounds on the wavelet coefficients $|\langle 2\arcsin \sqrt{f}, \psi_{j,k} \rangle|$. Let $j(x_0)$ be as in Lemma \ref{lem.wav_decay_small_arcsin}, setting $j(x_0) = \infty$ if $f(x_0)\in \{0,1\}$. Lemma \ref{lem.wav_decay_small_arcsin} yields that for any $\psi_{j,k}$ with $j\geq j(x_0)$ and $x_0\in \supp(\psi_{j,k})$,
\begin{equation*}
|d_{j,k}| = \left| \int_0^1 2\arcsin \sqrt{f(t) }\psi_{j,k}(t) dt \right| \leq C(\psi,\beta)\frac{R}{\sqrt{f(x_0)(1-f(x_0))}} 2^{-\frac{j}{2}(2\beta+1)}.
\end{equation*}
Conversely, for the low-frequency coefficients $j<j(x_0)$ with $\psi_{j,k}(x_0) \neq 0$, Lemma \ref{lem.wav_decay_small_arcsin} yields $|d_{j,k}| \leq C(\psi,\beta) \sqrt{R} \, 2^{-\frac{j}{2}(\beta+1)}$.

Note that we have exactly the same wavelet bounds as in the Poisson case in Theorem \ref{thm.est_link=sqrt}, except with $f(x_0)$ replaced by $f(x_0)(1-f(x_0))$. Therefore arguing as in Theorem \ref{thm.est_link=sqrt}, we can bound \eqref{eq.Besov_embed} as
\begin{align*}
\left| \widehat{h}(x_0) - h(f(x_0)) \right| \lesssim \min \left(  (\log n/n)^{\beta/(2\beta+2)} , \left[ f(x_0)(1-f(x_0))\right]^{-1/(4\beta+2)} (\log n/n)^{\beta/(2\beta+1)} \right) .
\end{align*}
Recall that $\widehat{f}(x_0) = \sin^2 [\tfrac{1}{2}\widehat{h}(x_0)]$. Since $|\widehat{h}(x_0) - h(f(x_0))| \rightarrow 0$, Taylor expansion yields
\begin{align*}
\left| \widehat{f}(x_0) - f(x_0) \right| & = \left| \sin^2 \tfrac 12 \widehat{h}(x_0) - \sin^2 \tfrac 12 h(f(x_0)) \right| \\
& = \left|\tfrac 12  (\widehat{h}(x_0) - h(f(x_0))) \sin h(f(x_0)) + \tfrac 12 (\widehat{h}(x_0) - h(f(x_0)))^2  \cos h(f(x_0)) \right| \\
& \quad \quad \quad \quad + O (|\widehat{h}(x_0) - h(f(x_0))|^3)  \\
& \lesssim |\widehat{h}(x_0) - h(f(x_0))| \max \left( \sin \tfrac 12 h(f(x_0)) \cos \tfrac 12 h(f(x_0)), |\widehat{h}(x_0) - h(f(x_0))|  \right) \\
& =  \max \left( \sqrt{f(x_0)(1-f(x_0))} |\widehat{h}(x_0) - h(f(x_0))| , |\widehat{h}(x_0) - h(f(x_0))|^2 \right) .
\end{align*}
Substituting in the bounds derived above yields
\begin{align*}
\left| \widehat{f}(x_0) - f(x_0) \right| \lesssim \max \Big( \Big( \frac{\log n}{n} \Big)^{\frac{\beta}{\beta+1}} ,  \Big( \frac{f(x_0)(1-f(x_0)) \log n}{n} \Big)^{\frac{\beta}{2\beta+1}}  \Big) .
\end{align*}
\end{proof}

\begin{proof}[Proof of Theorem \ref{thm.lower_bound_arcsin}]
The lower bound is proved in the same way as for the Poisson case in Theorem \ref{thm.lower_bound_sqrt} by considering a two hypothesis argument and verifying $(i)$-$(iii)$ for the parameter space $\{f: \|f\|_{\mH_B^\beta} \leq R, \KL(f,f_n^*)\leq 1\}$ and the rate $r_{n,\beta}^B.$ Consider the two cases:
\begin{align*}
	(A): \quad \min(f_n^*(x_0), 1-f_n^*(x_0)) > n^{-\frac{\beta}{\beta+1}} ,  \quad \quad \quad (B): \quad \min(f_n^*(x_0), 1-f_n^*(x_0)) \leq n^{-\frac{\beta}{\beta+1}} .
\end{align*}

$(A)$: Suppose firstly that $f_n^*(x_0) \leq 1/2$ and consider the hypotheses \eqref{eq.two_hyp_large}. It is shown in Theorem \ref{thm.lower_bound_sqrt} that these functions have $\|\cdot\|_{\mathcal{H}^\beta}$-norm at most $R$ and and this extends to $\|\cdot\|_{\mH_B^\beta}$ with only minor modifications. For the Kullback-Leibler divergence, we have using $(\arcsin x)' = (1-x^2)^{-1/2},$ $h_n^\beta \| K \|_\infty \rightarrow 0$, and Lemma \ref{lem.local_func_size_arcsin}, that for $n$ large enough, 
\begin{align*}
	KL (P_{f_{0,n}}, P_{f_{1,n}}) & = 	2n \int_0^1 \left( \arcsin \sqrt{f_n^*(x)} - \arcsin \sqrt{f_n^*(x) + R h_n^\beta K((x-x_0)/h_n)} \right)^2 dx \\
	& =2 n h_n \int_{-1}^1  \left( \arcsin \sqrt{f_n^*(x_0+h_nu)} - \arcsin \sqrt{f_n^*(x_0+h_nu) + R h_n^\beta K(u)} \right)^2 du \\
	& \leq 2n h_n \int_{-1}^1 \frac{\big( \sqrt{f_n^*(x_0+h_nu)} - \sqrt{f_n^*(x_0+h_nu) + R h_n^\beta K(u)} \big)^2}{1 - \max_{t\in[-1,1]} f_n^*(x_0+h_nt) - R h_n^\beta \|K\|_\infty }  du\\
	& \leq \frac{4nR^2 h_n^{2\beta+1} \| K \|_2^2 }{f_n^*(x_0) (1-f_n^*(x_0))} \leq 8R^2 c_0^{2\beta+1} \| K \|_2^2 \leq 1 
\end{align*}
for $c_0 \leq aR^{-1/\beta} \wedge (8R^2 \| K \|_2^2)^{-1/(2\beta+1)}$ and $a$ as in Lemma \ref{lem.local_func_size_arcsin}. This verifies $(i)$ and $(iii)$. For the pointwise distance, $|f_{0,n}(x_0)- f_{1,n}(x_0)| = R c_0^\beta K(0) (f_n^*(x_0) /n)^{\frac{\beta}{2\beta+1}} \gtrsim r_{n,\beta}^B(f_{0,n}(x_0)) \asymp r_{n,\beta}^B(f_{1,n}(x_0))$, thereby establishing $(ii)$. If $f_n^*(x_0) > 1/2$, then the same argument holds for the modified hypothesis $f_{1,n}^-(x) = f_n^*(x) - R h_n^\beta K((x-x_0)/h_n)$ with $h_n = c_0 ((1-f_n^*(x_0))/n)^{1/(2\beta+1)}$.

$(B)$: Suppose $f(x_0)\leq 1/2$ and consider hypotheses \eqref{eq.two_hyp_small}. It only remains to show the required bound on the Kullback-Leibler divergence. Using the usual change of variable, that $\arcsin$ is continuously differentiable near 0 and that $\sqrt{\cdot}$ is $1/2$-H\"older continuous,
\begin{align*}
KL (P_{f_{0,n}}, P_{f_{1,n}}) & = 2n h_n \int_{-1}^1 \left(\arcsin \sqrt{f_n^*(x_0+h_nu)} - \arcsin \sqrt{f_n^*(x_0+h_nu) + R h_n^\beta K(u)} \right)^2 du \\
& \leq 2n h_n \int_{-1}^1 \frac{\big| \sqrt{f_n^*(x_0+h_nu)} - \sqrt{f_n^*(x_0+h_nu) + R h_n^\beta K(u)} \big|^2}{1-\max_{t\in[-1,1]} f_n^*(x_0+h_nt) -R h_n^\beta \|K\|_\infty} du \\
& \leq \frac{4nRh_n^{\beta+1}\|K\|_1}{1-f_n^*(x_0)} \leq 8Rc_0^{\beta+1}\|K\|_1 \leq 1
\end{align*}
for $c_0 \leq (8R\|K\|_1)^{-1/(\beta+1)}$. If $f(x_0) > 1/2$, then the proof follows similarly for hypothesis $f_{1,n}^-(x) = f_n^*(x) - R h_n^\beta K((x-x_0)/h_n)$ with $h_n = c_n n^{-1/(\beta+1)}$, noting that $\sqrt{\cdot}$ and $\arcsin$ are respectively continuously differentiable and $1/2$-H\"older near 1. This completes the proof.
\end{proof}

\begin{proof}[Proof of Theorem \ref{thm.est_link=log}]
By Lemma 1.1.1 of \cite{csorgo1981}, for any $\eta >0$ there exists $C(\eta)>0$ such that for all $u>0$,
\begin{align*}
\P \Big( \sup_{\frac{1}{2n}\leq x\leq 1-\frac{1}{2n}} \sqrt{n}\big|W_{x+\frac{1}{2n}} - W_{x-\frac{1}{2n}}\big| \geq u\sqrt{2\log n}  \Big) \leq C(\eta) n^{1-\frac{2u^2}{2+\eta}} .
\end{align*}
Letting $\eta = \sigma -1>0$ and $u=(\sigma+1)/2 > 1$,
\begin{align*}
\P \Big( \sup_{\frac{1}{2n}\leq x\leq 1-\frac{1}{2n}} |\epsilon_n(x)| \geq \tfrac 12(\sigma+1) \sqrt{2\log n}  \Big) \leq C(\sigma) n^{-\tfrac{\sigma-1}{2}} .
\end{align*}
For the boundary values $x\in[0,1/(2n)]$ or $x\in[1-1/(2n),1]$, noting that $\text{var}(\epsilon_n(x)) \leq \text{var}(\sqrt{n}(W_{x+1/(2n)}-W_0))$, one can again apply Lemma 1.1.1 of \cite{csorgo1981} to obtain the same bound. We may thus restrict to the event $B_n = \{ \sup_{x\in[0,1]} |\epsilon_n(x)| \leq \tfrac{1}{2}(\sigma+1) \sqrt{2\log n} \}$, which has probability at least $1-C(\sigma) n^{-\tfrac{\sigma-1}{2}}$. We study the behaviour of $Z_n$ on $B_n$.

Consider $x_0 \in [1/(2n),1-1/(2n)]$ and note that the following argument can be easily modified to apply to $x_0 \in [0,1/(2n)] \cup [1-1/(2n),1]$. Suppose $f(x_0) \geq M_n n^{-\beta}$, where $M_n = C e^{(3\sigma+1)\sqrt{\log n}} \rightarrow \infty$ and $C>0$. It follows that $1/(2n) \leq a (f(x_0)/R)^{1/\beta}$ and so applying Lemma \ref{lem.local_func_size}, $|\log f(t) - \log f(x_0)| = \log (1 + (f(t)-f(x_0))/f(x_0)) = O(1)$ for any $|x_0-t| \leq 1/(2n)$. Consequently on $B_n$,
\begin{align*}
Z_n(x_0) & \geq 2^{-1/2}\log f(x_0) - \frac{n}{\sqrt{2}}\int_{x_0-\frac{1}{2n}}^{x_0+\frac{1}{2n}} \left| \log f(x_0) - \log f(t) \right| dt - \frac{\sigma+1}{2} \sqrt{2\log n}  \\
& \geq -\frac{\beta \log n}{\sqrt{2}} + 2^{-1/2}\log M_n - O(1) - \frac{\sigma+1}{2} \sqrt{2\log n}  \geq -\frac{\beta \log n}{\sqrt{2}} + \sigma\sqrt{2\log n} ,
\end{align*}
for $C>0$ large enough. Suppose now $f(x_0) \leq c e^{(\sigma-1)\sqrt{\log n}} n^{-\beta}$, where $0 < c < 1/2$, and that $Z_n(x_0) \geq - 2^{-1/2}\beta \log n + \sigma \sqrt{2\log n}$. Then
\begin{align*}
n \int_{x_0-\frac{1}{2n}}^{x_0+\frac{1}{2n}} \log f(t) dt = \sqrt{2}\big(Z_n(x_0) - \epsilon_n (x_0)\big) \geq -\beta \log n + (\sigma-1) \sqrt{\log n} ,
\end{align*}
and consequently there exists $t^*$ with $|x_0 - t^*| \leq 1/(2n)$ and $f(t^*) \geq e^{(\sigma-1)\sqrt{\log n}} n^{-\beta}$. Noting that $|x_0 - t^*| \leq a (f(t^*)/R)^{1/\beta}$, we can apply Lemma \ref{lem.local_func_size} to yield $f(x_0) \geq f(t^*) - |f(t^*) - f(x_0)| \geq \frac{1}{2}f(t^*)$, which is a contradiction. Hence for such $f(x_0)$, it holds that $Z_n(x_0) \leq -2^{-1/2}\beta \log n + \sigma \sqrt{2\log n}$. In conclusion, we have shown that there exist $C,c>0$ such that
\begin{equation}
\begin{split}
B_n \subset & \{ Z_n(x) \geq -\frac{\beta \log n}{\sqrt{2}} + \sigma \sqrt{2\log n} \text{ for all $x$ with } f(x) \geq C e^{(3\sigma+1)\sqrt{\log n}} n^{-\beta} \} \\
& \quad  \cap \{ Z_n(x) < -\frac{\beta \log n}{\sqrt{2}} + \sigma \sqrt{2\log n} \text{ for all $x$ with } f(x) \leq c e^{(\sigma-1)\sqrt{\log n}} n^{-\beta} \}.
\label{eq.threshold_event_log}
\end{split}
\end{equation}

We now analyze the estimator $\widehat{h}$ as in Theorem \ref{thm.est_link=sqrt}, working on the event $A_n \cap B_n$, where $A_n$ is given in \eqref{eq.wavelet_event}. Fix $x_0 \in (0,1)$ and again let $h(f) = \sum_{j,k} d_{j,k} \psi_{j,k}$ (here $h(f) = 2^{-1/2}\log f$). By \eqref{eq.threshold_event_log}, we need only consider $f(x_0) > ce^{(\sigma-1)\sqrt{\log n}} n^{-\beta}$, otherwise the estimator $\widehat{f}(x_0)$ equals 0 and the loss is $f(x_0)$. We again bound \eqref{eq.Besov_embed}, but with $(d_{j,k})$, $(\widehat{d}_{j,k})$ corresponding to model \eqref{eq.mod_link=log}. We establish the decay of the wavelet coefficients $|\langle 2^{-1/2} \log f, \psi_{j,k} \rangle|$. Let $j(x_0)$ be as in Lemma \ref{lem.wav_decay_log} (recall $f(x_0) > 0$). Lemma \ref{lem.wav_decay_log} yields that for any $\psi_{j,k}$ with $j\geq j(x_0)$ and $x_0\in \supp(\psi_{j,k})$,
\begin{equation}
|d_{j,k}| = \left| \int_0^1 2^{-1/2} \log f(t) \psi_{j,k}(t) dt \right| \leq C(\psi,\beta)\frac{R}{f(x_0)} 2^{-\frac{j}{2}(2\beta+1)}.
\label{eq.wavelet_bound_log}
\end{equation}

Denote by $J(x_0)$ the smallest integer $j$ such that $CRf(x_0)^{-1}2^{-\frac j2 (2\beta+1)} \leq \frac{\tau}{2}\sqrt{\log n/n}$; it follows that $2^{J(x_0)} \asymp f(x_0)^{-\frac{1}{\beta+1/2}} (n/\log n)^{\frac{1}{2\beta+1}}$. Bounding \eqref{eq.Besov_embed} as in Theorem \ref{thm.est_link=sqrt} but using the wavelet estimate \eqref{eq.wavelet_bound_log} instead (which we may do since $2^{J(x_0)} \gg 2^{j(x_0)}$ for $f(x_0) > ce^{(\sigma-1)\sqrt{\log n}} n^{-\beta}$),
\begin{align*}
\left| \widehat{h}(x_0) - h(f(x_0)) \right| \lesssim 2^{J(x_0)/2} \sqrt{\log n/n} \lesssim f(x_0)^{-\frac{1}{2\beta+1}} (\log n/n)^{\frac{\beta}{2\beta+1}}.
\end{align*}
Since $f(x_0) \geq ce^{(\sigma-1)\sqrt{\log n}} n^{-\beta}$, we have $| \widehat{h}(x_0) - h(f(x_0))| = o(1)$ so that by the exponential expansion
\begin{align*}
\left| e^{\sqrt{2}\widehat{h}(x_0)} - e^{\sqrt{2}h(f(x_0))} \right| = f(x_0) \left| e^{\sqrt{2}\widehat{h}(x_0) - \sqrt{2}h(f(x_0))} - 1 \right| \lesssim f(x_0)^\frac{2\beta}{2\beta+1} (\log n/n)^{\frac{\beta}{2\beta+1}} .
\end{align*}
Finally, for $c e^{(\sigma-1)\sqrt{\log n}} n^{-\beta} \leq f(x_0) \leq M_n n^{-\beta}$ (the set for which we may or may not threshold on $B_n$), the above rate is of smaller order than $f(x_0)$, the error from using the 0 estimator.
\end{proof}

\begin{proof}[Proof of Theorem \ref{thm.lower_bound_log}]
We again verify $(i)$-$(iii)$ as in Theorem \ref{thm.lower_bound_sqrt}. Consider the two cases
\begin{align*}
	(A): \quad f_n^*(x_0) > n^{-\beta} , \quad \quad \quad \quad (B): \quad f_n^*(x_0) \leq n^{-\beta} .
\end{align*}
$(A)$: Consider the hypotheses \eqref{eq.two_hyp_large} but with $h_n = c_0(f_n^*(x_0)^2/n)^{1/(2\beta+1)}$. It is shown in Theorem \ref{thm.lower_bound_sqrt} that these functions have $\|\cdot\|_{\mH^\beta}$-norm at most $R$, so we need only verify the Kullback-Leibler bound. If $c_0>0$ is taken sufficiently small, one can apply Lemma \ref{lem.local_func_size} to obtain $f_n^*(x) > f_n^*(x_0)/2$ for all $x\in[x_0-h_n,x_0+h_n]$. Consequently, using that $h_n \rightarrow 0$ and $\log (1+y)\leq y,$
\begin{align*}
	KL (P_{f_{0,n}}, P_{f_{1,n}})
	& = \frac{n}{4} \int_0^1 \left( \log f_{0,n}(x)-\log f_{1,n}(x)\right)^2 dx \\
	&= 	\frac{nh_n}{4} \int_{-1}^1 \log^2 \left( 1 + \frac{Rh_n^\beta K(u) }{f_n^*(x_0+h_nu)} \right) du \\
	& \leq \frac{nh_n}{4} \int_{-1}^1  \left( \frac{Rh_n^\beta K(u)}{f_n^*(x_0+h_nu)} \right)^2 du \\
	& \leq \frac{R^2 n h_n^{2\beta+1}\|K\|_2^2}{f_n^*(x_0)^2}
	= R^2 c_0^{2\beta+1}\|K\|_2^2 \leq 1
\end{align*}
for $c_0 \leq aR^{-1/\beta}\wedge (R^2 \|K\|_2^2)^{-1/(2\beta+1)}$ and $a$ as in Lemma \ref{lem.local_func_size}, which verifies $(i)$ and $(iii)$. For $(ii)$, it holds that $|f_{0,n}(x_0)- f_{1,n}(x_0)| = R c_0^\beta K(0) (f_n^*(x_0)^2 /n)^{\frac{\beta}{2\beta+1}} \gtrsim r_{n,\beta}^G(f_{0,n}(x_0)) \asymp r_{n,\beta}^G(f_{1,n}(x_0)).$ This completes the proof for $(A).$

$(B)$: Consider hypotheses \eqref{eq.two_hyp_small}, but with $h_n = c_0 f_n^*(x_0)^{1/\beta}$. It only remains to show the required bound on the Kullback-Leibler divergence. Using Lemma \ref{lem.local_func_size} as in case $(A)$ and with the usual change of variable,
\begin{align*}
	\KL(P_{f_{0,n}}, P_{f_{1,n}}) 
	&= \frac{nh_n}{4} \int_{-1}^1  \log^2 \left( 1 + \frac{Rh_n^\beta K(u)}{f_n^*(x_0 + h_nu)} \right) du  \\
	& \leq \frac{nh_n}{4} \int _{-1}^1 \log^2 \left(1 + \frac{R(c_0 f_n^*(x_0)^{1/\beta})^\beta K(u)}{f_n^*(x_0)/2} \right) du \\
	&\leq \frac{nh_n}{4} \int_{-1}^1 \log^2 \left( 1 + 2Rc_0^\beta K(u) \right) du .
\end{align*}
The last integral is finite and can be made arbitrarily small by taking $c_0>0$ small enough. For such a $c_0$, the right-hand side can thus be made smaller than $nf_n^*(x_0)^{1/\beta} \leq 1$ as required. This completes the proof.
\end{proof}

\begin{proof}[Proof of Theorem \ref{thm.plug_in_bd}]
Throughout the proof we write $L_n =\tfrac{\log n}n$, $g:=Kf,$ $\widehat{g}:=\widehat{Kf}=Y^\delta$ and let $C=C(\beta,R)$ denote a generic constant, which may change from line to line. 

First consider the link function $h=2\sqrt{\cdot}.$ In this case the local rate is $\widetilde r_{n,\beta}^P(g(t)),$ with $\widetilde r_{n,\beta}^P(u) = L_n^{\frac{\beta}{\beta+1}}\vee (uL_n)^{\frac{\beta}{2\beta+1}}.$ Notice that $u\mapsto\widetilde r_{n,\beta}^P(u)$ is monotone increasing, that $u\leq \widetilde r_{n,\beta}^P(u)$ if $u\leq L_n^{\frac{\beta}{\beta+1}}$ and that $\widetilde r_{n,\beta}^P(u)\leq u$ if $u\geq L_n^{\frac{\beta}{\beta+1}}.$  The last inequality can be extended so that for any $Q>0,$ 
\begin{align}
	\widetilde r_{n,\beta}^P(u) \leq \frac{u}{Q} \vee Q^{\frac{\beta}{\beta+1}} L_n^{\frac{\beta}{\beta+1}},
	\label{eq.techn_res_tilde_rn_proof}
\end{align}
which can be verified by treating the cases $u\gtrless Q^{\frac{2\beta+1}{\beta+1}} L_n^{\frac{\beta}{\beta+1}}$ separately. We complete the proof by separately checking the four cases
\begin{align*}
	(I): \quad &g(t)\geq L_n^{\frac{\beta}{\beta+1}} \quad \text{and}  \ \ \ \widehat {g}(t)\geq L_n^{\frac{\beta}{\beta+1}}, \\
	(II): \quad &g(t)\leq L_n^{\frac{\beta}{\beta+1}} \quad \text{and}  \ \ \ \widehat {g}(t)\geq L_n^{\frac{\beta}{\beta+1}}, \\
	(III): \quad &g(t)\geq L_n^{\frac{\beta}{\beta+1}} \quad \text{and}  \ \ \ \widehat {g}(t)\leq L_n^{\frac{\beta}{\beta+1}}, \\
	(IV): \quad &g(t)\leq L_n^{\frac{\beta}{\beta+1}} \quad \text{and}  \ \ \ \widehat {g}(t)\leq L_n^{\frac{\beta}{\beta+1}}.
\end{align*}
It is enough to show $C^{-1}\widehat{g}(t) \leq g(t) \leq C\widehat{g}(t)$, since then the rates are also equivalent up to constants. Recall that we are working on the event $\{|\widehat g(t)-g(t)|\leq C\widetilde r_{n,\beta}^P (g(t))\}$.

{\it (I):} Obviously, $\widehat g(t) \leq g(t) + C \widetilde r_{n,\beta}^P(g(t)) \leq Cg(t).$ For the other direction, using \eqref{eq.techn_res_tilde_rn_proof} with $Q=2C$ yields $g(t) \leq \widehat g(t) + C \widetilde r_{n,\beta}^P(g(t))\leq \widehat g(t) + \tfrac 12 g(t) \vee (2C)^{\frac{\beta}{\beta+1}}\widehat g(t).$ This proves that $C^{-1}\widehat{g}(t) \leq g(t) \leq C\widehat{g}(t)$ in case $(I).$

{\it (II)}: Since $\widehat g(t) \leq g(t) + C \widetilde r_{n,\beta}^P(g(t)) \leq CL_n^{\frac{\beta}{\beta+1}}$ we can conclude that $C^{-1}\widetilde r_{n,\beta}^P(\widehat{g}(t)) \leq \widetilde r_{n,\beta}^P(g(t))\leq C \widetilde r_{n,\beta}^P(\widehat{g}(t)).$

{\it (III):} Using \eqref{eq.techn_res_tilde_rn_proof} again, $g(t) \leq \widehat g(t) + C \widetilde r_{n,\beta}^P(g(t))\leq \widehat g(t)+ \tfrac 12 g(t) \vee (2C)^{\frac{\beta}{\beta+1}} L_n^{\frac{\beta}{\beta+1}}$ and we can then argue as in $(II).$

$(IV):$ Obviously $\widetilde r_{n,\beta}^P(g(t))=L_n^{\frac{\beta}{\beta+1}}=\widetilde r_{n,\beta}^P(\widehat{g}(t)).$

Together $(I)-(IV)$ yield the assertion of the theorem for the Poisson case.

Next, consider the link function $h(x)=2\arcsin\sqrt{x}$ from Section \ref{sec.Bernoulli_case}. Observe that for any $\beta>0$ there exists a finite integer $N=N(\beta)$ such that $\sup_{\|Kf\|_{\mH^\beta}\leq R} \|Kf-\widehat{Kf}\|_{\infty}\leq \tfrac 14,$ for all $n\geq N.$ For sufficiently large $n,$ we can thus assume that either $Kf(t), \widehat{Kf}(t)\in [0,3/4]$ or $Kf(t), \widehat{Kf}(t)\in [1/4,1].$ The rate function $\widetilde r_{n,\beta}^B$ can be linked to the rate function in the Poisson case via 
\begin{align*}
	&\frac{1}{2}\widetilde r_{n,\beta}^P(u)  \leq \widetilde r_{n,\beta}^B(u) \leq \widetilde r_{n,\beta}^P(u), \quad \quad \ u\in [0,3/4],\\
	&\frac{1}{2}\widetilde r_{n,\beta}^P(1-u)  \leq \widetilde r_{n,\beta}^B(u) \leq \widetilde r_{n,\beta}^P(1-u), \quad \quad \ u\in [1/4,1].
\end{align*}
Due to the above, the assertion for the link function $h(x)=2\arcsin\sqrt{x}$ follows from the Poisson case.

Finally, let us prove the result for the link function $h(x)=2^{-1/2}\log x.$ Observe that $\widehat{g}(t)\leq g(t) +C\widetilde r_{n,\beta}^G(g(t))\leq (1+C) g(t).$ If $g(t)\geq M_nn^{-\beta}$ with $M_n$ as in the proof of Theorem \ref{thm.est_link=log}, then $\widetilde r_{n,\beta}^G(g(t))\leq \xi_n g(t)$ with $\xi_n \rightarrow 0$ and so $g(t) \leq 2\widehat{g}(t)$ for all sufficiently large $n.$ This completes the proof for the link function $h(x)=2^{-1/2}\log x.$
\end{proof}

\section{Technical results}
\label{sec.technical_results}

The first three lemmas are Lemmas 5.1, 5.2 and Proposition 3.2 of \cite{RaySchmidt-Hieber2015c}. To determine the size of the wavelet coefficients of $h\circ f$ we have a standard bound.
\begin{lem}
\label{lem.wav_decay_gen}
Suppose that the wavelet function is $S$-regular. If $f\in C^\beta([0,1])$ for $0 < \beta<S,$ then there exists a function $g$ with $\|g\|_\infty \leq 1,$ such that for any $x_0\in (0,1),$ 
\begin{align*}
	\left| \int f(x) \psi_{j,k}(x) dx \right| 
	\leq \frac{1}{\lfloor \beta \rfloor!} \left| \int \big[ f^{(\lfloor \beta \rfloor)}(x_0+g(y) (y-x_0))- f^{(\lfloor \beta \rfloor)}(x_0) \big] (y-x_0)^{\lfloor \beta \rfloor} \psi_{j,k} (y) dy \right|.
\end{align*}
\end{lem}

\begin{lem}
\label{lem.local_func_size}
Suppose that $f \in \mH^\beta$ with $\beta>0$ and let $a = a(\beta)>0$ be any constant satisfying $(e^a-1) + a^\beta / (\lfloor \beta \rfloor!) \leq 1/2.$ Then for 
\begin{align*}
	|h| \leq a \left( \frac{|f(x)|}{\|f\|_{\mH^\beta}}\right)^{1/\beta},
\end{align*}
we have
\begin{equation*}
|f(x+h) - f(x)| \leq \frac{1}{2} |f(x)| ,
\end{equation*}
implying in particular, $|f(x)|/2 \leq |f(x+h)| \leq 3|f(x)|/2$.
\end{lem}

The previous lemma controls the local fluctuations of a function in $\mH^\beta$ and allows one to obtain the following bound on the decay of the wavelet coefficients.

\begin{lem}
\label{lem.wav_decay_small}
Suppose that $\psi$ is $S$-regular and that $f\in\mathcal{H}^\beta$ for $0 < \beta < S$. Then
\begin{equation*}
|\langle \sqrt{f}, \psi_{j,k} \rangle| \leq C(\psi,\beta) \|f\|_{\mathcal{H}^\beta}^{1/2} \, 2^{-\frac{j}{2}(\beta+1)} .
\end{equation*}
For $x_0\in[0,1]$, let $j(x_0)$ be the smallest integer satisfying $2^{j(x_0)}\geq |\supp(\psi)|a^{-1} (\|f\|_{\mH^\beta}/f(x_0))^{1/\beta}$ where $a=a(\beta)$ is the constant in Lemma \ref{lem.local_func_size}. Then for any wavelet $\psi_{j,k}$ with $j\geq j(x_0)$ and $x_0\in \supp(\psi_{j,k})$,
\begin{equation*}
|\langle \sqrt{f}, \psi_{j,k} \rangle| \leq C(\psi,\beta)\frac{ \|f\|_{\mH^\beta} }{\sqrt{f(x_0)}} 2^{-\frac{j}{2}(2\beta+1)}.
\end{equation*}
\end{lem}

We have analogous results for the function space $\mH_B^\beta$.

\begin{lem}
\label{lem.local_func_size_arcsin}
Suppose that $f \in \mH_B^\beta$ with $\beta>0$ and let $a = a(\beta)>0$ be any constant satisfying $(e^a-1) + a^\beta / (\lfloor \beta \rfloor!) \leq 1/2.$ Then for 
\begin{align*}
	|h| \leq a \left( \frac{\min(f(x),1-f(x))}{\|f\|_{\mH_B^\beta}}\right)^{1/\beta},
\end{align*}
we have
\begin{align*}
|f(x+h) - f(x)| = |(1-f(x)) - (1-f(x+h))| \leq \frac{1}{2} \min (f(x), 1-f(x)) ,
\end{align*}
implying in particular
\begin{align*}
\frac{1}{4}f(x)(1-f(x)) \leq f(x+h)(1-f(x+h)) \leq \frac{9}{4} f(x)(1-f(x)) .
\end{align*}
\end{lem}

\begin{proof}
Recall that for $f\in \mathcal{H}_B^\beta$, we have $|f^{(j)}(x)| \leq \|f\|_{\mH_B^\beta}^{j/\beta} \min (f(x),1-f(x))^{\frac{\beta-j}{\beta}}$ for all $x \in [0,1]$. The proof then follows as in Lemma 5.2 of \cite{RaySchmidt-Hieber2015c}.
\end{proof}

The next result shows that if $f\in \mH_B^\beta$, then the function $\arcsin \sqrt{f}$ satisfies a H\"older-type condition with exponent $\beta$ and locally varying H\"older constant. The proof relies heavily on Fa\`a di Bruno's formula, which generalizes the chain rule to higher derivatives \cite{Johnson2002}:
\begin{equation}
\frac{d^k}{dx^k} h(f(x)) = \sum_{(m_1,...,m_k) \in \mathcal{M}_k} \frac{k!}{m_1!...m_k!} h^{(m_1+...+m_k)}(f(x)) \prod_{j=1}^k \left( \frac{f^{(j)}(x)}{j!} \right)^{m_j} ,
\label{eq.FaaDiBruno}
\end{equation}
where $\mathcal{M}_k$ is the set of all $k$-tuples of non-negative integers satisfying $\sum_{j=1}^k j m_j = k$. We can relate the derivatives appearing in \eqref{eq.FaaDiBruno} to $f$ using the seminorm $|\cdot|_{\mH_B^\beta}$.

\begin{lem}\label{lem.arcsin_lip}
For $\beta>0$, there exists a constant $C(\beta)$ such that for all $f \in \mathcal{H}_B^\beta$, $0\leq k < \beta$ and $x,y\in[0,1]$, 
\begin{equation*}
|(\arcsin \sqrt{f(x)})^{( \lfloor \beta \rfloor )} - (\arcsin \sqrt{f(y)})^{(\lfloor \beta \rfloor)} | \leq \frac{C(\beta)\|f\|_{\mathcal{H}_B^\beta} |x-y|^{\beta - \lfloor \beta \rfloor} }{\min (\sqrt{f(x)(1-f(x))}, \sqrt{f(y)(1-f(y))}) } ,
\end{equation*}
and
\begin{align}
\left| \frac{d^k}{dx^k} \arcsin \sqrt{f(x)} \right| \leq C(\beta) \|f\|_{\mH_B^\beta}^{k/\beta} \left[ f(x)(1-f(x))\right]^{1/2-k/\beta}.
\label{eq.deriv_bds_arcsin}
\end{align}
Moreover, if $f\in [\varepsilon, 1-\varepsilon]$ for some $0<\varepsilon< 1/2,$
\begin{align*}
\|\arcsin \sqrt{f}\|_{\mathcal{H}_B^\beta} \leq \frac{C(\beta)}{\sqrt{\varepsilon}} \|f\|_{\mathcal{H}_B^\beta}.
\end{align*}
\end{lem}

\begin{rem}\label{rem.arcsin_lip}
It is actually proved below that the first inequality of Lemma \ref{lem.arcsin_lip} holds with $|f|_{C^\beta} + |f|_{\mH^\beta} + |1-f|_{\mH^\beta}$ instead of $\|f\|_{\mH_B^\beta}$.
\end{rem}

\begin{proof}
The proof follows the same approach as that of Lemma 5.3 in \cite{RaySchmidt-Hieber2015c}. For convenience write $h(x) = \arcsin \sqrt{x}$, $R = |f|_{C^\beta} + |f|_{\mH^\beta} + |1-f|_{\mH^\beta} \leq \|f\|_{\mathcal{H}_B^\beta}$ and $\delta(x) = \min (f(x),1-f(x))$ and without loss of generality assume $\delta(y)\leq \delta(x).$ Let $C(\beta)$ be a generic $\beta$-dependent constant, which may change from line to line. 

We first prove the result for $\beta \in (0,1].$ Noting that $h'(x) = 1/\sqrt{x(1-x)}$ is decreasing on $[0,1/2]$ and increasing on $[1/2,1]$ and applying the mean-value theorem,
\begin{align*}
| h(f(x)) - h(f(y))| \leq \max_{f(x) \wedge f(y) \leq t \leq f(x) \vee f(y)} h'(t) |f(x) - f(y)| \lesssim  \frac{R |x-y|^\beta}{\sqrt{f(y)(1-f(y))}} .
\end{align*}
Consider now $\beta > 1$ and write $k = \lfloor \beta \rfloor$ (the subsequent arguments also hold for all $k<\beta$ with certain modifications). We consider separately the two cases where $|x-y|$ is small and large.

Suppose first that $|x-y| \leq a(\delta(x)/R)^{1/\beta}$ with $a$ as in Lemma \ref{lem.local_func_size_arcsin}. By Lemma \ref{lem.local_func_size_arcsin} we have $\delta(y)/4 \leq \delta(x) \leq 9\delta(y)/4$ and this will be used frequently without mention below. We shall establish the result by proving a H\"older bound for each of the summands in Fa\`a di Bruno's formula \eqref{eq.FaaDiBruno} individually. Fix a $k$-tuple $(m_1,...,m_k) \in \mathcal{M}_k$ and write $M:=\sum_{j=1}^k m_j$. By the triangle inequality, $|  h^{(M)}(f(x))  \prod_{j=1}^k \big( f^{(j)}(x) \big)^{m_j} -  h^{(M)}(f(y)) \prod_{j=1}^k \big( f^{(j)}(y) \big)^{m_j}  \big|\leq (I)+(II)$ with
\begin{align*}
	&(I):=\big |  \big( h^{(M)}(f(x)) - h^{(M)}(f(y))  \big) \prod_{j=1}^k \big( f^{(j)}(x) \big)^{m_j}  \big|,  \\
	&(II):=\big| h^{(M)}(f(y)) \big(  \prod_{j=1}^k \big( f^{(j)}(x) \big)^{m_j} -  \prod_{j=1}^k \big( f^{(j)}(y) \big)^{m_j}  \big)  \big| .
\end{align*}
We bound $(I)$ and $(II)$ separately.

We first require some additional results. Note that $h^{(M)}(t) = \tfrac{1}{2} \tfrac{d^{M-1}}{dt^{M-1}} (t(1-t))^{-1/2}$ since $M\geq1$. Applying \eqref{eq.FaaDiBruno} to this last expression (with $f(t) = t-t^2$) and using the binomial theorem for $1-2t = (1-t)-t$ yields for $t \in (0,1)$,
\begin{align}
h^{(M)} (t) = \sum C_{p,q}\frac{(1-2t)^p}{(t(1-t))^{p+q+\frac 12}}
= \sum \sum_{r=0}^p C_{p,q,r} \frac{1}{t^{r+q+\frac 12}(1-t)^{p+q-r+\frac 12}}
\label{arcsin lemma eq1}
\end{align}
for suitable constants $C_{p,q}, C_{p,q,r},$ with the unspecified sums taken over all non-negative integers $p,q$ satisfying $p+2q=M-1$ (these are the tuples of the form $(p,q,0,...,0)\in \mathcal{M}_{M-1}$ in \eqref{eq.FaaDiBruno} -- since $\tfrac{d^3}{dt^3} (t-t^2)=0$, all the other terms in \eqref{eq.FaaDiBruno} equal zero). By the definition of $\mathcal{H}_B^\beta$ and since $\sum jm_j=k$,
\begin{equation}
\big| \prod_{j=1}^k (f^{(j)}(x))^{m_j} \big |  \leq  \prod_{j=1}^k R^\frac{jm_j}{\beta} \delta(x)^\frac{(\beta-j)m_j}{\beta} = R^\frac{k}{\beta} \delta(x)^{M - k/\beta} .
\label{arcsin lemma eq2}
\end{equation}
Arguing as in (5.5) and (5.6) of \cite{RaySchmidt-Hieber2015c}, for any integer $\ell,$
\begin{align}
	|f(x)^{-\ell -\frac{1}{2}}- f(y)^{-\ell -\frac{1}{2}}|
	\leq CR^{1-\frac{k}{\beta}} |f(y)|^{-\ell - \frac 32 +\frac{k}{\beta}} |x-y|^{\beta-k}
	\label{eq.inv_mom_f_diff_bd}
\end{align}
and the same bound holds if $f$ is replaced by $1-f,$ that is,
\begin{align}
	|(1-f(x))^{-\ell -\frac{1}{2}}- (1-f(y))^{-\ell -\frac{1}{2}}|
	\leq CR^{1-\frac{k}{\beta}} |1-f(y)|^{-\ell - \frac 32 +\frac{k}{\beta}} |x-y|^{\beta-k}.
	\label{eq.inv_mom_1-f_diff_bd}
\end{align}

{\it $(I)$:} Rewriting $h^{(M)}$ using \eqref{arcsin lemma eq1} and controlling the difference between the terms $f(x)^{-r-q-\frac 12}(1-f(x))^{-p-q+r-\frac 12}$ and $f(y)^{-r-q-\frac 12}(1-f(y))^{-p-q+r-\frac 12}$ with \eqref{eq.inv_mom_f_diff_bd} and \eqref{eq.inv_mom_1-f_diff_bd} gives
\begin{align*}
	\big| h^{(M)}(f(x)) - h^{(M)}(f(y))  \big|
	\leq CR^{1-\frac{k}{\beta}}\sum \sum_{r=0}^p\frac{f(y)^{\frac{k}{\beta}-1} + (1-f(y))^{\frac k{\beta} -1}}{f(y)^{r+q+\frac 12}(1-f(y))^{p+q-r+\frac 12}} |x-y|^{\beta-k}.
\end{align*}
Recall that $r+q \leq p+q\leq  M-1.$ Together with \eqref{arcsin lemma eq2}, we finally obtain
\begin{align*}
	(I) \leq C R \delta(y)^{-1/2} |x-y|^{\beta-k}.
\end{align*}

{\it $(II)$:}  We recall the following bound from \cite{RaySchmidt-Hieber2015c} (specifically (5.7) is shown to bound the second term of (5.3) in that paper), adjusted slightly since $f\in\mH_B^\beta$ rather than $\mH^\beta$ in the present case,
\begin{align*}
\left| \left(  \prod_{j=1}^k \big( f^{(j)}(x) \big)^{m_j} -  \prod_{j=1}^k \big( f^{(j)}(y) \big)^{m_j}  \right)  \right| \leq CR \delta(y)^{M-1} |x-y|^{\beta-k} .
\end{align*}
Together with the first equality in \eqref{arcsin lemma eq1} and since $p+q\leq M-1$, this proves that $(II) \leq C R \delta(y)^{-1/2} |x-y|^{\beta-k}.$

The bounds for $(I)$ and $(II)$ complete the proof in the case $|x-y| \leq a(\delta(x)/R)^{1/\beta}$.

The first equality in \eqref{arcsin lemma eq1} implies that $h^{(M)}(f(x)) \leq C\delta(x)^{\frac 12 -M}.$ Substituting this and \eqref{arcsin lemma eq2} into Fa\`a di Bruno's formula \eqref{eq.FaaDiBruno} gives
\begin{equation}
\begin{split}
\left| \frac{d^k}{dx^k} h(f(x)) \right| & \lesssim R^{k/\beta} \delta(x)^{1/2-k/\beta}.
\end{split}
\label{arcsin lemma eq7}
\end{equation}
For $|x-y| > a(\delta(x)/R)^{1/\beta},$ this yields
\begin{align*}
|h(f(x))^{(k)} - h(f(y))^{(k)}| & \leq CR^{k/\beta} (\delta(x)^{1/2-k/\beta} + \delta(y)^{1/2-k/\beta}) \\
& \leq \frac{CR}{\sqrt{\delta(y)}} \left( \frac{\delta(x)}{R} \right)^\frac{\beta-k}{\beta} \leq \frac{CR|x-y|^{\beta-k} }{\sqrt{f(y)(1-f(y))}}
\end{align*}
as required. This completes the proof of the first statement.

Inequality \eqref{eq.deriv_bds_arcsin} follows directly from \eqref{arcsin lemma eq7}, since this last expression also holds for all $0\leq k<\beta$. For the last assertion of the lemma, suppose now $f \in [\varepsilon, 1-\varepsilon]$ for $\varepsilon\in(0,1/2).$ The first statement of the lemma yields that $\|\arcsin \sqrt{f}\|_{C^\beta} \leq C \|f\|_{\mathcal{H}_B^\beta}/\sqrt{\varepsilon}$. By \eqref{eq.deriv_bds_arcsin},
\begin{align*}
|h(f(x))^{(j)}| \leq C\|f\|_{\mH_B^\beta}^{j/\beta} \delta(x)^{1/2-j/\beta} (\delta(x)/\varepsilon)^{j/(2\beta)} = C (\|f\|_{\mH_B^\beta}   /\sqrt{\varepsilon})^\frac{j}{\beta} (\sqrt{\delta(x)})^{\frac{\beta-j}{\beta}}
\end{align*}
for all $1\leq j<\beta,$ implying that $|\arcsin\sqrt{f}|_{\mathcal{H}^\beta} \leq C\|f\|_{\mathcal{H}_B^\beta} /\sqrt{\varepsilon}$.
\end{proof}

Using the previous result, we have a direct analogue of Lemma \ref{lem.wav_decay_small} for the Bernoulli case.

\begin{lem}
\label{lem.wav_decay_small_arcsin}
Suppose that $\psi$ is $S$-regular and that $f \in \mathcal{H}_B^\beta$ for $0 < \beta < S$. Then
\begin{equation*}
|\langle \arcsin \sqrt{f} , \psi_{j,k} \rangle | \leq  C(\psi,\beta) \|f\|_{\mathcal{H}_B^\beta}^{1/2} \, 2^{-\frac{j}{2}(\beta+1)}.
\end{equation*}
For $x_0\in[0,1]$, let $j(x_0)$ be the smallest integer satisfying
\begin{align*}
2^{j(x_0)}\geq \frac{|\supp(\psi)|}{a} \left( \frac{\|f\|_{\mH_B^\beta} }{\min(f(x_0),1-f(x_0))} \right)^{1/\beta},
\end{align*}
where $a=a(\beta)$ is the constant in Lemma \ref{lem.local_func_size_arcsin}. Then for any wavelet $\psi_{j,k}$ with $j\geq j(x_0)$ and $x_0\in \supp(\psi_{j,k})$,
\begin{equation*}
|\langle \arcsin \sqrt{f}, \psi_{j,k} \rangle| \leq C(\psi,\beta)\frac{ \|f\|_{\mH_B^\beta} }{\sqrt{f(x_0)(1-f(x_0))}} 2^{-\frac{j}{2}(2\beta+1)}.
\end{equation*}
\end{lem}

\begin{proof}
For $0<\eta\leq1/4$, let $f_\eta = (f + \eta)(1-2\eta)$ and note that $f_\eta \in [c\eta, 1-c\eta]$ for some $c>0$ independent of $\eta$. We have $|f_\eta|_{C^\beta} \leq |f|_{C^\beta}$ and, using that the $|\cdot|_{\mH^\beta}$-seminorm of a positive constant function is 0, $|f_\eta|_{\mH^\beta} \leq  |f|_{\mH^\beta}$, $|1-f_\eta|_{\mH^\beta} \leq  |1-f|_{\mH^\beta}$. Using Lemma \ref{lem.arcsin_lip} and Remark \ref{rem.arcsin_lip} then yields $|\arcsin \sqrt{f_\eta}|_{C^\beta} \leq C\|f\|_{\mH_B^\beta}/\sqrt{\eta}$. Consider $f(x) \leq 1/2$. Note that $\arcsin (t)$ is continuously differentiable for $0 \leq t \leq 3/4$ with bounded derivative $(1-t^2)^{-1/2}$. Thus, using also that $\sqrt{\cdot}$ is $1/2$-H\"older continuous and that $\| f \|_\infty \leq 1$,
\begin{align*}
|\arcsin \sqrt{f(x)} - \arcsin \sqrt{f_\eta(x)}| \lesssim |\sqrt{f(x)} - \sqrt{f_\eta(x)}| \lesssim |\eta(1-2f(x)) - 2\eta^2|^{1/2}  \lesssim \sqrt{\eta} .
\end{align*}
For $f(x) \geq 1/2$ an identical bound holds using that $t\mapsto \arcsin(t)$ is $1/2$-H\"older continuous for $1/4\leq t \leq 1$ and that $\sqrt{\cdot}$ is continuously differentiable with bounded derivative on this interval. Using the inequality above and Lemma \ref{lem.wav_decay_gen},
\begin{align*}
|\langle \arcsin \sqrt{f} , \psi_{j,k} \rangle| & \leq \left| \int \arcsin \sqrt{f_\eta} \psi_{j,k} \right|  +  \left|  \int (\arcsin \sqrt{f_\eta} -\arcsin \sqrt{f}) \psi_{j,k} \right| \\
& \lesssim   \eta^{-1/2} \|f\|_{\mH_B^\beta} 2^{-\frac{j}{2}(2\beta+1)} + \sqrt{\eta} 2^{-j/2} .
\end{align*}
Balancing these gives $\eta = 2^{-j\beta}\|f\|_{\mH_B^\beta}$ and thus the first result.

Assume that $f(x_0)\leq 1/2$ (the case $f(x_0)>1/2$ is similar). We show that for all $j \geq j(x_0)$, the support of any $\psi_{j,k}$ with $\psi_{j,k}(x_0) \neq 0$ is contained in the set $\{ x : f(x_0) /2 \leq f(x) \leq 3f(x_0)/2 \}$. To see this observe that for any such $\psi_{j,k}$ it holds that $|\supp(\psi_{j,k})| \leq 2^{-j}|\supp(\psi)| \leq a (f(x_0)/\|f\|_{\mH_B^\beta})^{1/\beta}$ and so applying Lemma \ref{lem.local_func_size_arcsin} yields that $|f(t)-f(x_0)| \leq f(x_0)/2$ for any $t \in \supp(\psi_{j,k})$. Using this and applying Lemma \ref{lem.arcsin_lip},
\begin{align*}
\|\arcsin \sqrt{f}\|_{\mH^\beta(\supp(\psi_{j,k}))} \leq C(\beta)\|f\|_{\mH_B^\beta} /\sqrt{f(x_0)/2},
\end{align*}
where the first norm refers to $\arcsin \sqrt{f}$ restricted to $\supp(\psi_{j,k})$ with the obvious modification of $\|\cdot\|_{\mH_B^\beta}$ to this set. Applying Lemma \ref{lem.wav_decay_gen} to such $(j,k)$ then yields the result.
\end{proof}

\begin{lem}
\label{lem.log_lip}
For $\beta>0$, there exists a constant $C(\beta)$ such that for all $f \in \mathcal{H}^\beta$, $0\leq k <\beta$ and $x,y\in[0,1]$,
\begin{equation*}
|(\log f(x))^{( \lfloor \beta \rfloor )} - (\log f(y))^{(\lfloor \beta \rfloor)} | \leq \frac{C(\beta)(|f|_{C^\beta} + |f|_{\mathcal{H}^\beta}) }{\min(f(x), f(y))} |x-y|^{\beta - \lfloor \beta \rfloor}
\end{equation*}
and 
\begin{equation*}
\left| \frac{d^k}{dx^k} \log f(x) \right| \leq C(\beta) \|f\|_{\mH^\beta}^{k/\beta} f(x)^{-k/\beta}.
\end{equation*}
Moreover, if $f \geq \varepsilon > 0$,
\begin{align*}
\|\log f\|_{\mathcal{H}^\beta} \leq \frac{C(\beta)}{\varepsilon} \|f\|_{\mathcal{H}^\beta}.
\end{align*}
\end{lem}

\begin{proof}
The proof follows as that of Lemma \ref{lem.arcsin_lip} above or Lemma 5.3 of \cite{RaySchmidt-Hieber2015c}, noting that $\log^{(r)}(x) = (-1)^{r-1} (r-1)! x^{-r}$ for $r \geq 1$.
\end{proof}

The large fluctuations of $\log x$ near 0 mean we can only obtain the second wavelet bounds of Lemmas \ref{lem.wav_decay_small} and \ref{lem.wav_decay_small_arcsin} in this case.

\begin{lem}
\label{lem.wav_decay_log}
Suppose that $\psi$ is $S$-regular and that $f\in\mathcal{H}^\beta$ for $0 < \beta < S$. For $x_0\in[0,1]$, let $j(x_0)$ be the smallest integer satisfying $2^{j(x_0)}\geq |\supp(\psi)|a^{-1} (\|f\|_{\mH^\beta}/f(x_0))^{1/\beta}$ where $a=a(\beta)$ is the constant in Lemma \ref{lem.local_func_size}. Then for any wavelet $\psi_{j,k}$ with $j\geq j(x_0)$ and $x_0\in \supp(\psi_{j,k})$,
\begin{equation*}
|\langle \log{f}, \psi_{j,k} \rangle| \leq C(\psi,\beta)\frac{ \|f\|_{\mH^\beta} }{f(x_0)} 2^{-\frac{j}{2}(2\beta+1)}.
\end{equation*}
\end{lem}

\begin{proof}
The proof follows as in Proposition 3.2 of \cite{RaySchmidt-Hieber2015c}.
\end{proof}

\begin{lem}\label{lem.mollifier}
There exists a non-negative, symmetric, infinitely differentiable function $K_0$ supported on $[-1,1]$ such that for any $\beta >0$ and all $x \in \R$,
\begin{align*}
|K_0^{(j)}(x)| \leq C(\beta,j)|K_0(x)|^\frac{\beta-j}{\beta}, \quad \quad \quad j = 1,...,\lfloor \beta \rfloor .
\end{align*}
\end{lem}

\begin{proof}
The function $K_0(x) = \exp (-1/(1-x^2) ) 1\{ |x| \leq 1 \}$ is symmetric, non-negative and infinitely differentiable. For the last condition, note that $K_0^{(j)}(x) = p_j(x) (1-x^2)^{-2j} K_0(x)$ for some polynomial $p_j$ of degree at most $2j$. Then for $x\in[-1,1]$,
\begin{align*}
|K_0^{(j)}(x)|^\frac{\beta}{\beta-j} = \left| \frac{p_j(x)}{(1-x^2)^{2j}} \right|^{\frac{\beta}{\beta-j}} e^{-\frac{j}{(\beta-j)(1-x^2)}} K_0(x) \leq C'(\beta,j) K_0(x).
\end{align*}
\end{proof}

\bibliographystyle{acm}    
\bibliography{bibhd}           

\end{document}